\documentclass[12pt]{article}
\usepackage[T1]{fontenc}
\usepackage[margin=0.8in]{geometry}
\usepackage{amsmath,amssymb,amsthm,mathtools,bm}
\usepackage{aliascnt}
\usepackage{booktabs,array}
\usepackage{float}
\usepackage[section]{algorithm}
\usepackage{algpseudocode}
\usepackage{enumitem}
\usepackage{graphicx}
\usepackage[numbers,sort&compress]{natbib}
\usepackage{microtype}
\usepackage{xcolor}
\usepackage{tikz}
\usetikzlibrary{arrows.meta,calc,patterns,positioning}
\usepackage[colorlinks=true,linkcolor=blue!55!black,citecolor=blue!55!black,urlcolor=blue!55!black]{hyperref}
\usepackage[nameinlink,capitalise,noabbrev]{cleveref}

\allowdisplaybreaks
\newtheorem{theorem}{Theorem}[section]
\newaliascnt{proposition}{theorem}
\newtheorem{proposition}[proposition]{Proposition}
\aliascntresetthe{proposition}
\newaliascnt{lemma}{theorem}
\newtheorem{lemma}[lemma]{Lemma}
\aliascntresetthe{lemma}
\newaliascnt{corollary}{theorem}

\aliascntresetthe{corollary}
\newaliascnt{remark}{theorem}

\aliascntresetthe{remark}
\newaliascnt{definition}{theorem}

\aliascntresetthe{definition}

\crefname{theorem}{theorem}{theorems}
\Crefname{theorem}{Theorem}{Theorems}
\crefname{proposition}{proposition}{propositions}
\Crefname{proposition}{Proposition}{Propositions}
\crefname{lemma}{lemma}{lemmas}
\Crefname{lemma}{Lemma}{Lemmas}
\crefname{corollary}{corollary}{corollaries}
\Crefname{corollary}{Corollary}{Corollaries}
\crefname{remark}{remark}{remarks}
\Crefname{remark}{Remark}{Remarks}
\crefname{definition}{definition}{definitions}
\Crefname{definition}{Definition}{Definitions}
\crefname{algorithm}{algorithm}{algorithms}
\Crefname{algorithm}{Algorithm}{Algorithms}

\newcommand{\R}{\mathbb R}

\newcommand{\E}{\mathbb E}
\newcommand{\Pp}{\mathbb P}
\newcommand{\ind}{\mathbf 1}

\newcommand{\cH}{\mathcal H}

\newcommand{\cQ}{\mathcal Q}
\newcommand{\op}{\mathrm{op}}
\newcommand{\clip}{\operatorname{clip}}
\newcommand{\Var}{\operatorname{Var}}

\definecolor{ed}{RGB}{225,0,100}

\title{\Large Improved Variance Estimation
in Homoskedastic Nonparametric Random-Design Regression via a Two-Scale Approach}

\author{
Edgar Dobriban, Rajarshi Mukherjee, James M. Robins, and  Zixiao Wang\footnote{
ED: Department of Statistics and Data Science, University of Pennsylvania.
RM and JMR: Department of Biostatistics, Harvard T.H. Chan School of Public Health.
JMR and ZW: Department of Epidemiology, Harvard T.H. Chan School of Public Health.
E-mail addresses: ED: \texttt{dobriban@wharton.upenn.edu};
RM: \texttt{ram521@mail.harvard.edu};
JMR: \texttt{robins@hsph.harvard.edu};
ZW: \texttt{zixiaowang@fas.harvard.edu}.}
}

\date{\today}

\begin{document}

\maketitle

\begin{abstract}
We study estimation of a constant conditional variance $\sigma^2$ in nonparametric regression with a $d$-dimensional random design. 
This is an important problem, and similar questions arise in causal inference. 
The regression function is $\beta_b$-H\"older smooth, the design density is $\beta_g$-H\"older smooth and bounded above and away from zero, and we consider the nonparametric regime $\beta_b>1$ and $d>4\beta_b$. Set
$\beta_g^\star=\beta_b(1-4\beta_b/d)/\{1+2\beta_b/d+8(\beta_b/d)^2\}$.
We give an estimator whose mean squared error is upper bounded by $Cn^{-4(\beta_b+1)/(d+4)}$
in the low-regularity regime 
when $0<\beta_g\leq\beta_g^\star$.
The low-regularity branch is based on a new two-scale construction: the covariate space is partitioned into cells, 
the local polynomial trend is projected out within each suitable cell, and the squared normalized contrast from one eligible close pair per cell is averaged across cells. 
In the high regularity regime 
when $\beta_g>\beta_g^\star$,
a higher-order influence function estimator of Robins, Li, Tchetgen Tchetgen, and van der Vaart (2008) provides the  rate $Cn^{-8\beta_b/(d+4\beta_b)}$. 
We also give an all-pairs ridge extension, which achieves the same two-scale rate, and evaluate the methods alongside a range of existing estimators in simulations.
\end{abstract}


\section{Introduction}

For an integer $d\geq1$, consider independent observations $(X_i,Y_i)$, $i=1,\ldots,n$, from the homoskedastic nonparametric regression model
\begin{equation}\label{eq:model-intro}
Y_i=f(X_i)+\varepsilon_i,
\qquad
\E(\varepsilon_i\mid X_i)=0,
\qquad
\E(\varepsilon_i^2\mid X_i)=\sigma^2.
\end{equation}
The main goal of this paper is to provide an estimator of $\sigma^2$ that improves on the close-pair benchmark without requiring a smooth design density. 
The covariate $X_i$ takes values in $[0,1]^d$. Its density $p$ is bounded above and away from zero and belongs to an isotropic H\"older ball of smoothness $\beta_g>0$.\footnote{The notation $\beta_b,\beta_g$ is borrowed from \cite{RobinsEtAl2008} for easing the comparison of results.}
The regression function $f$ belongs to a
H\"older ball of smoothness $\beta_b>1$. 
The conditional error law may depend on $X_i$, subject to the common conditional variance in \eqref{eq:model-intro} and a uniform fourth-moment bound.

This is an important and fundamental problem in itself, because the estimation of variance is required for forming prediction sets and prediction bands for the outcomes.
 Moreover this problem is also closely related to (and in fact a special case of) important problems in causal inference.
 Specifically, it is related to treatment effect estimation in observational studies \cite{RobinsEtAl2008}.

In a different setting, 
when the regression function is 
more rough than what we consider here, i.e.,
$0<\beta_b<1$ and $d>4\beta_b$, \citet{RobinsEtAl2008} constructed a close-pair estimator with mean squared error of order $n^{-8\beta_b/(d+4\beta_b)}$. This is faster than the order $n^{-4\beta_b/d}$ attainable under an equally spaced fixed design -- see e.g. \citet{MunkEtAl2005,WangEtAl2008,CaiLevineWang2009} and references therein. 
This suggests that random design can make variance estimation easier than its fixed-design analogue. It motivates asking whether such an improvement also holds when $\beta_b>1$ without assuming density smoothness.

Indeed, the case $\beta_b>1$ has been identified as substantially more difficult. \citet[Section~6 and note~49]{RichardsonRotnitzky2014} recorded a question posed by Robins as to whether the preceding random-design rate of $n^{-8\beta_b/(d+4\beta_b)}$ remains attainable above smoothness one. \citet{AronowLopatto2026} recently proved that it is not uniformly attainable when $\beta_b>1$ and $d>4\beta_b$ over the larger class of design densities constrained only to be bounded above and away from zero. 
The  close-pair benchmark considered below uses only Lipschitz continuity of $f$ and has optimized worst-case mean squared error of order $n^{-8/(d+4)}$; see \Cref{prop:raw-close-pair}. 
These estimators and the broader literature are discussed in \Cref{sec:related}.

This paper 
provides an improved estimator in the above regime.
Let
$\beta_g^\star=\beta_b(1-4\beta_b/d)/\{1+2\beta_b/d+8(\beta_b/d)^2\}$ 
be a threshold for the density regularity parameter $\beta_g$. 
In the low regularity regime $0<\beta_g\leq\beta_g^\star$, we construct a new two-scale 
estimator with mean squared error upper bounded by $O(n^{-4(\beta_b+1)/(d+4)})$.
The new two-scale estimator
has faster rate than the close-pair method.
In the higher regularity regime of $\beta_g>\beta_g^\star$, 
a higher-order influence function estimator from \citet{RobinsEtAl2008}
attains the comparison rate $O(n^{-8\beta_b/(d+4\beta_b)})$ under the additional uniform moment condition in \eqref{eq:hoif-uniform-moment-condition}.

The estimator we describe below uses a two-scale construction that goes beyond the one-scale close-pair construction and does not require density estimation as in the cited higher-order influence function construction \citep{RobinsEtAl2008}. 
This two-scale construction is motivated by an interpolation between polynomial approximation of $\beta_b>1$-smooth regression functions and geometric localization needed by the one-scale close pair estimator that can only utilize Lipschitz-ness. 

Specifically, we partition $[0,1]^d$ into $\Theta(n)$ cells of diameter $H_n\asymp n^{-1/d}$.
We then search for a suitable contrast vector, whose squared inner product with outcomes will eventually yield desired estimator(s). 
Within a suitable cell, an empirical projection of suitable contrast vectors removes the degree-$\ell$ local polynomial trend, 
where $\ell=\lceil\beta_b\rceil-1$. 
We then anchor the projected contrast at one pair whose normalized within-cell separation is at most $\varepsilon_n$. 
The square of the resulting unit-norm contrast has conditional bias of order at most $H_n^{2\beta_b}\varepsilon_n^2$. 
It turns out that there are order $n\varepsilon_n^d$ usable cells that yield desirable contrast vectors, implying a variance of order $(n\varepsilon_n^d)^{-1}$ that can be balanced with the squared-bias contribution of order $H_n^{4\beta_b}\varepsilon_n^4$. Choosing $\varepsilon_n\asymp n^{-(d-4\beta_b)/\{d(d+4)\}}$ balances these terms and yields $n^{-4(\beta_b+1)/(d+4)}$ rate for the mean squared error in the low regularity regime; 
improving on the close-pair benchmark \citep{RobinsEtAl2008}.

The main technical ingredients are the projected close-pair bias bound in \Cref{lem:pe-bias bound} and the Poissonization--de-Poissonization argument in \Cref{lem:pe-availability}, which shows that a suitable screening leaves order $n\varepsilon_n^d$ usable cells with exponentially high probability in the original fixed-size sample. We first formulate an ideal variational estimator and then replace it by a numerically feasible estimator. 
The appendix develops a numerically stable all-pairs ridge extension.
We also provide experiments comparing with a variety of standard estimators,
the close pairs estimator of \citet{RobinsEtAl2008},
a specialization of the double-cross-fit doubly robust expected conditional covariance estimator of \citet{newey2018cross,mcclean2026double},
and
an all-pairs kernel $U$-statistic motivated by \citet{MullerSchickWefelmeyer2003,DuSchick2009,ShenEtAl2020}.

\section{Model and main result}\label{sec:model}
Our models are described by H\"{o}lder smoothness classes.
Fix an open set $U\subset\R^d$ containing $[0,1]^d$. For $s>0$, set $\ell_s=\lceil s\rceil-1$ and $\alpha_s=s-\ell_s\in(0,1]$. For a function $g:U\to\R$ with the required derivatives, define
\begin{equation*}
[D^{\ell_s}g]_{\alpha_s;U}
=
\max_{|\nu|=\ell_s}
\sup_{x\neq y\in U}
\frac{|\partial^\nu g(x)-\partial^\nu g(y)|}{\|x-y\|^{\alpha_s}}
\end{equation*}
and
$\|g\|_{C^s(U)}=
\sum_{|\nu|\leq\ell_s}\|\partial^\nu g\|_{L^\infty(U)}
+[D^{\ell_s}g]_{\alpha_s;U}$.
Specifically, when $s$ is an integer, we will use 
$C^s:=C^{s-1,1}$ as a convention. 
Finally, for $L>0$, define $\cH^s(L;U):=\{g:U\to\R:\|g\|_{C^s(U)}\leq L\}$ to be the H\"{o}lder ball over $U$ of smoothness $s$ and radius $L$. 

A parameter is $\theta=(p,f,\sigma^2,Q)$, where $p$ is a design density on $[0,1]^d$, $f$ is a regression function, $\sigma^2>0$, and $Q=(Q_x)_{x\in[0,1]^d}$ is a Markov kernel for the conditional error law. We observe independent copies $(X_i,Y_i)$ satisfying $X_i\sim p$, $Y_i=f(X_i)+\varepsilon_i$, and $\varepsilon_i\mid X_i=x\sim Q_x$.
To define our model, further fix constants
\begin{equation*}
0<\underline p<1<\overline p<\infty,
\qquad
L\geq1,
\qquad
0<\underline v\leq\overline v<\infty,
\qquad
M_4>\overline v^{\,2},
\end{equation*}
where $L\geq1$ is necessary for a nonempty density class, and let $\Theta(\beta_b,\beta_g,d)$ be the collection of parameters satisfying:
\begin{enumerate}[label=(\roman*),leftmargin=2.2em]
\item $p$ is the restriction to $[0,1]^d$ of a function in $\cH^{\beta_g}(L;U)$ and $\underline p\leq p(x)\leq\overline p$ for Lebesgue-almost every $x\in[0,1]^d$;
\item $f$ is the restriction to $[0,1]^d$ of a function in $\cH^{\beta_b}(L;U)$;
\item for $p(x)dx$-almost every $x$,
\begin{equation*}
\int u\,Q_x(du)=0,
\qquad
\int u^2\,Q_x(du)=\sigma^2,
\qquad
\int u^4\,Q_x(du)\leq M_4,
\end{equation*}
with $\underline v\leq\sigma^2\leq\overline v$.
\end{enumerate}
The conditional error distribution may otherwise depend arbitrarily on $x$. Since a density is defined only up to a null set, in the following we will fix a version of $p$ for which the bounds in condition (i) hold pointwise. Finally, although the model also depends on $L,U,\underline{p},\overline{p},\underline{v},\overline{v},M_4$, we suppress them from the notation to keep the notation light.

Results on estimating $\sigma^2$ in the model above depend intricately on the smoothness parameters $\beta_b,\beta_g$. Whereas a detailed literature exists for large $\beta_g$ \citep{RobinsEtAl2008}, our goal is to address low to no smoothness regimes corresponding to $\beta_g$. Indeed, if one sets
\begin{equation}\label{eq:robins-density-threshold}
\beta_g^\star
=
\frac{\beta_b(1-4\beta_b/d)}
{1+2\beta_b/d+8(\beta_b/d)^2},
\end{equation}
then, under the additional conditions specified below, the high-regularity comparison for $\beta_g>\beta_g^\star$ uses the higher-order influence function estimator of \citet{RobinsEtAl2008}, specialized from the expected conditional covariance problem by taking $A=Y$, the two conditional mean nuisances equal to $f$, and $H_1=1$, $H_2=H_3=-Y$, and $H_4=Y^2$. Let $\widetilde\sigma_{n,\mathrm{HOIF}}^2$ denote this special case of \citet[Theorems~4.5--4.6]{RobinsEtAl2008}  
and let $\widehat\sigma_{n,\mathrm{HOIF}}^2=\clip_{[\underline v,\overline v]}(\widetilde\sigma_{n,\mathrm{HOIF}}^2)$, where $\clip_{[a,b]}(z)=\min\{b,\max\{a,z\}\}$. For this comparison we additionally require the uniform integrated bias and variance bounds in \eqref{eq:hoif-uniform-moment-condition}. 
In the low-regularity regime, let $\widetilde\sigma_{n,\mathrm{LR}}^2$ be the feasible witness estimator that we will describe below in \eqref{eq:feasible-lr-estimator}, with $\varepsilon_n$ chosen as in \eqref{eq:lr-optimal-bandwidth}. 
Define the smoothness-dependent estimator
\begin{equation}\label{eq:sharp-estimator-definition}
\widehat\sigma_{n,\mathrm{sharp}}^2
=
\begin{cases}
\widehat\sigma_{n,\mathrm{HOIF}}^2,&\beta_g>\beta_g^\star,\\
\widetilde\sigma_{n,\mathrm{LR}}^2,&0<\beta_g\leq\beta_g^\star.
\end{cases}
\end{equation}

Our main result is as follows.

\begin{theorem}[Upper bound and conditional HOIF comparison]\label{thm:main}
Let $\beta_b>1$, $\beta_g>0$, and $d>4\beta_b$. Under conditions (i)--(iii), the feasible witness estimator, with the screening constants and bandwidth of \Cref{prop:ideal_est_lr2}, satisfies \eqref{eq:feasible-estimator-rate} for every $\beta_g>0$.
$$
    \sup_{\theta\in\Theta(\beta_b,\beta_g,d)}
    \E_\theta(\widetilde\sigma_{n,\mathrm{LR}}^2-\sigma^2)^2
    \leq Cn^{-4(\beta_b+1)/(d+4)}.
$$
If, in addition, the high-regularity HOIF construction satisfies \eqref{eq:hoif-uniform-moment-condition} uniformly over $\Theta(\beta_b,\beta_g,d)$ whenever $\beta_g>\beta_g^\star$, then the estimator in \eqref{eq:sharp-estimator-definition} satisfies, for all sufficiently large $n$,
\begin{equation}\label{eq:sharp-smooth-piecewise}
\sup_{\theta\in\Theta(\beta_b,\beta_g,d)}
\E_\theta
\bigl(\widehat\sigma_{n,\mathrm{sharp}}^2-\sigma^2\bigr)^2
\leq
\begin{cases}
Cn^{-8\beta_b/(d+4\beta_b)},
&\beta_g>\beta_g^\star,\\[2mm]
Cn^{-4(\beta_b+1)/(d+4)},
&0<\beta_g\leq\beta_g^\star,
\end{cases}
\end{equation}
\end{theorem}
Since $\widehat\sigma^2_{n,\mathrm{HOIF}}$ has already been detailed in \citet{RobinsEtAl2008}, here we focus on
the low-regularity regime. 

\section{Low-Regularity Regime \texorpdfstring{$0<\beta_g\leq\beta_g^\star$}{(0 < beta-g <= beta-g-star)}}\label{sec:est_low_reg}

In this section, we motivate the construction of estimators that attain the second line of the upper bound in \Cref{thm:main}, corresponding to the low-regularity regime $0<\beta_g\leq\beta_g^\star$. Throughout, let $\ell=\lceil\beta_b\rceil-1$ and $q=q_{d,\ell}=\binom{d+\ell}{\ell}$. Fix a basis $\psi_1,\ldots,\psi_q$ of the polynomials on $\R^d$ of total degree at most $\ell$, and write $\psi(u)=(\psi_1(u),\ldots,\psi_q(u))^\top$. Because the basis is fixed and polynomial, it is bounded and Lipschitz on $[0,1]^d$.

\subsection{Background and Intuition}\label{sec:intuition}

To gain intuition 
for estimating $\sigma^2$ 
one can start from the existing literature on difference-based estimation; going back to \citet{von1941mean}, and elaborated in successive works such as \citet{stone1984asymptotically,CaiLevineWang2009}, etc. 
Here, one takes squared contrasts of observations with nearby indices. 
As we argue below, 
existing indexing strategies 
are not directly optimal in our setting. 
Nevertheless, one can still study squared contrasts built from
selected observations whose realized covariates are suitably useful. 
The construction is summarized visually in \Cref{fig:feasible-witness-schematic}.

\begin{figure}
\centering
\begin{tikzpicture}[
  x=1cm,y=1cm,font=\small,>=Latex,
  sample/.style={circle,fill=black!45,inner sep=1.2pt},
  pair/.style={circle,draw=blue!70!black,fill=blue!18,
               line width=0.8pt,minimum size=5.2pt,inner sep=0pt},
  box/.style={draw=black!28,rounded corners=2pt,fill=black!2,
              inner xsep=4pt,inner ysep=3pt,align=center}
]
\begin{scope}
  \node[font=\bfseries] at (1.50,3.58) {(a) Screen cells};
  \fill[blue!6] (0.75,1.50) rectangle (1.50,2.25);
  \fill[blue!6] (2.25,0.75) rectangle (3.00,1.50);
  \draw[black!35,line width=0.35pt] (0,0) grid[step=0.75] (3,3);
  \draw[black,line width=0.8pt] (0,0) rectangle (3,3);
  \draw[blue!70!black,line width=0.95pt] (0.75,1.50) rectangle (1.50,2.25);
  \draw[blue!70!black,line width=0.95pt] (2.25,0.75) rectangle (3.00,1.50);
  \foreach \p in {(0.28,0.35),(0.56,1.08),(0.22,2.63),(0.92,0.30),
                   (1.22,0.60),(1.72,0.28),(2.22,0.45),(2.63,0.34),
                   (0.90,1.72),(1.16,1.91),(1.36,2.10),(1.82,1.62),
                   (2.45,1.03),(2.68,1.17),(2.84,1.38),(0.37,2.18),
                   (1.82,2.54),(2.33,2.12),(2.72,2.67)}
    \node[sample] at \p {};
  \node[blue!70!black,fill=white,inner sep=1pt,font=\scriptsize]
    at (1.13,2.43) {good $Q$};
  \node[box,text width=3.30cm,font=\scriptsize] at (1.50,-0.68)
    {$J_n\asymp n$, $\operatorname{diam}(Q)\asymp H_n$\\[-1pt]
     $N_Q=q+2$, Gram screen,\\[-1pt]
     $\mathcal E_Q(\varepsilon_n)\neq\varnothing$};
\end{scope}

\draw[->,line width=0.8pt] (3.24,1.50)--(3.90,1.50)
  node[midway,above,font=\scriptsize] {zoom};

\begin{scope}[xshift=4.15cm]
  \node[font=\bfseries] at (1.75,3.58) {(b) Project the pair};
  \fill[black!1] (0,0) rectangle (3.50,3.00);
  \draw[black,line width=0.8pt] (0,0) rectangle (3.50,3.00);
  \foreach \p in {(0.35,0.42),(0.70,1.28),(0.60,2.46),(1.12,0.76),
                   (1.34,2.22),(1.94,0.40),(2.24,1.20),(2.56,2.56),
                   (2.96,0.76),(3.18,1.84)}
    \node[sample] at \p {};
  \coordinate (p1) at (1.34,2.22);
  \coordinate (p2) at (1.68,2.40);
  \node[pair] at (p1) {};
  \node[pair] at (p2) {};
  \draw[blue!70!black,line width=1.05pt] (p1)--(p2);
  \draw[<->,blue!70!black,line width=0.55pt]
    ($(p1)+(0.00,0.25)$)--($(p2)+(0.00,0.25)$)
    node[midway,above,font=\scriptsize] {$\leq\varepsilon_n$};
  \node[anchor=west,blue!70!black,fill=white,inner sep=1pt,
        font=\scriptsize] at (1.85,2.15) {selected $e_Q$};
  \node[box,text width=3.65cm,font=\footnotesize] at (1.75,-0.68)
    {$\displaystyle
      \widetilde{\mathbf a}_Q=
      \frac{R_Qd_{e_Q}}
      {\sqrt{d_{e_Q}^{\top}R_Qd_{e_Q}}}$\\[-1pt]
     $\Phi_Q^\top\widetilde{\mathbf a}_Q=0$,
     $\|\widetilde{\mathbf a}_Q\|_2=1$};
\end{scope}

\draw[->,line width=0.8pt] (7.90,1.50)--(8.58,1.50);

\begin{scope}[xshift=8.82cm]
  \node[font=\bfseries] at (2.65,3.58) {(c) Average};
  \node[font=\footnotesize] at (2.65,2.98)
    {$Z_Q=(\widetilde{\mathbf a}_Q^\top Y_Q)^2$ for each good cell};
  \node[box,minimum width=1.05cm] (z1) at (0.65,2.18) {$Z_{Q_1}$};
  \node[box,minimum width=1.05cm] (z2) at (2.00,2.18) {$Z_{Q_2}$};
  \node at (3.16,2.18) {$\cdots$};
  \node[box,minimum width=1.05cm] (zm) at (4.42,2.18) {$Z_{Q_m}$};
  \node[box,text width=4.75cm,minimum height=1.02cm] (avg) at (2.65,0.95)
    {$\displaystyle
      \widetilde\sigma_{n,\mathrm{LR}}^2
      =\frac1m\sum_{r=1}^m Z_{Q_r}$\\[-1pt]
     $m=|\mathcal T_n(\varepsilon_n)|\gtrsim n\varepsilon_n^d$};
  \draw[->,line width=0.65pt] (z1.south)--($(avg.north)+(-1.35,0)$);
  \draw[->,line width=0.65pt] (z2.south)--($(avg.north)+(-0.45,0)$);
  \draw[->,line width=0.65pt] (zm.south)--($(avg.north)+(1.35,0)$);
\end{scope}
\end{tikzpicture}
\caption{Construction of our first estimator, shown schematically in two dimensions. Partition the design space into cells of diameter $\approx H_n$ and retain cells containing $q+2$ observations that pass certain screens. In each retained cell, choose the lexicographically first close pair $e_Q$, project its contrast onto $\mathcal N(\Phi_Q^\top)$, and normalize. This exactly annihilates degree-$\ell$ polynomial trends while preserving localization at the fine normalized scale $\varepsilon_n$. Finally, average these squared projected contrasts over the retained cells. }
\label{fig:feasible-witness-schematic}
\end{figure}
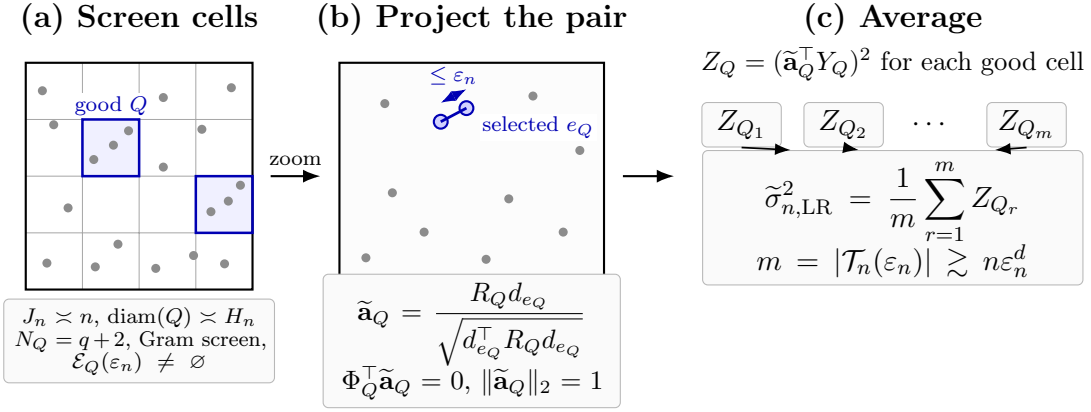

We elaborate on this below.

\begin{itemize}[leftmargin=2em]
    \item First consider generic realized covariate points $X_1,\ldots,X_k\in[0,1]^d$, with the precise value of $k$ to be specified later. Let the corresponding realized outcomes be $Y_r$, $r=1,\ldots,k$. 
    Consider a squared-difference estimator obtained from a contrast vector $\mathbf a=(a_1,\ldots,a_k)^\top\in\R^k$ through $Z_{\mathbf a}=(\sum_{r=1}^k a_rY_r)^2$. Then
    \begin{equation*}
        \E(Z_{\mathbf a}\mid X_1,\ldots,X_k)
        =\sigma^2\|\mathbf a\|_2^2
        +\left(\sum_{r=1}^k a_rf(X_r)\right)^2.
    \end{equation*}
    By normalizing so that $\|\mathbf a\|_2=1$, 
    the bias (conditional on the design) is captured by $(\sum_{r=1}^k a_rf(X_r))^2$. 
    At this point, it is natural to expand $f$ through its Taylor polynomial around a suitable point. To make the bias small, however, one must choose this point and ensure that the lower-order Taylor terms do not contribute much.

    \item Based on the first observation, suppose we try to find a collection of Taylor polynomial annihilators. The number $k$ should be related to the number of independent polynomial terms that we aim to annihilate. It is therefore natural to consider the space $\mathcal P_{\leq\ell}$ of all polynomials of total degree at most $\ell=\lceil\beta_b\rceil-1$, whose dimension is $q=\binom{d+\ell}{\ell}$.

    To find an annihilator, it is natural to look for covariate points that are close. One can first divide $[0,1]^d$ into cells of side length of order $H_n\to0$ and then consider the $m$ random observations falling in a cell $Q$. Denote their covariates by $X_{i_r}$, $r=1,\ldots,m$, and their within-cell coordinates (normalized so that the cell is rescaled to $[0,1]^d$) by $U_{i_r}\in[0,1]^d$. Let
    $    \Phi_Q=
        \begin{pmatrix}
        \psi(U_{i_1})^\top\\
        \vdots\\
        \psi(U_{i_m})^\top
        \end{pmatrix}$.
    One can then seek a cell-specific contrast $\mathbf a_Q$ satisfying $\Phi_Q^\top\mathbf a_Q=0$, where the subscript $Q$ emphasizes that the choice is specific to the cell, and average the squared contrast 
    $Z_{\mathbf a_Q}$ over suitable cells. A generic unit annihilator gives a conditional-bias bound of order $H_n^{2\beta_b}$. In the regime $nH_n^d\lesssim1$, averaging one squared contrast per usable cell gives a variance bound of order $\{n(nH_n^d)^q\}^{-1}$.
    Balancing squared bias and variance yields a mean squared-error upper bound of order
    $      n^{-4\beta_b(q+1)/(dq+4\beta_b)}$,
    which is slower than the desired $n^{-4(\beta_b+1)/(d+4)}$ in our regime. 
    
    This suboptimality stems from the trade-off faced by the cell size $H_n$. A smaller $H_n$ improves the bias, but worsens the variance, due to requiring at least $q+1$ observations to fall in the same small cell. 
    This suggests that one should not always demand that at least $q+1$ observations fall in the region corresponding to the finest scale. Instead of using an arbitrary annihilator,
    one can seek a suitably optimal one.

    \item At this point, it is also natural to ask what would happen if we did not demand at least $q+1$ observations, but instead used the more classical choice $k=2$ and averaged $(Y_i-Y_j)^2/2$ over pairs whose covariates are close. The close-pair contrast $e_i-e_j$ annihilates constants, but not nonconstant polynomial trends. Consequently, when $\beta_b>1$, it can exploit only the Lipschitz regularity of $f$ and saturates at the standard close-pair rate $n^{-8/(d+4)}$, see Section \ref{cp}. Thus the close-pair estimator does not benefit from smoothness  beyond one.

    \item To obtain mean squared error better than the preceding two benchmarks, one can address the failure point of each and combine their strengths. The estimator considered below can be viewed as an interpolation between pair-based estimation and the degree-$\ell$ polynomial-annihilation for uniformly close points.

    One way to see this is to consider the vector $v=e_i-e_j\in\R^m$ corresponding to a close-pair contrast. Although $v$ does not exactly annihilate nonconstant polynomials, it is an approximate annihilator because
    \begin{equation*}
        \Phi_Q^\top v=\psi(U_i)-\psi(U_j),
    \end{equation*}
    which is $O(\|U_i-U_j\|)$ by the Lipschitz continuity of the basis. Thus, one only needs to find a vector that is close to $v$ but exactly belongs to $\mathcal N(\Phi_Q^\top)$. The remaining points used to construct this correction need not be close to the pair, as long as they lie in the same cell.

    This suggests a two-scale strategy. First create cells of diameter of order $H_n$. Within each cell, find a closer pair to create an initial candidate vector $v$. Then use the remaining points in the cell to 
    slightly correct $v$, keeping the favorable variance properties of the close-pair contrast. This yields the constrained optimization problem of finding $\mathbf a_Q\in\mathcal N(\Phi_Q^\top)$ that minimizes $\|\mathbf a_Q-v\|_2$. A solution on typical, well-behaved cells yields the desired bound. As the analysis below shows, such cells essentially require full column rank and uniform conditioning of $\Phi_Q$, together with enough observations in $Q$; sufficiently many such cells occur with high probability.

    \item There is another closely related way to motivate the estimator. 
    In principle, it would help 
    to use
    a unit annihilator that minimizes the worst-case conditional bias $(\sum_{r=1}^m a_rf(X_{i_r}))^2$ over $f\in\cH^{\beta_b}(L;U)$. 
    One can show that this optimum is attained at some vector $\mathbf a_Q^\star$. It is then natural to work with $Z_{\mathbf a_Q^\star}$ over suitable cells. The success of this strategy depends on 
    the optimized contrast achieving a better bias than the generic $H_n^{2\beta_b}$ bound.
    However, the vector $\mathbf a_Q^\star$ is only a theoretical construct and it remains to be seen how to implement this empirically. 

    The close-pair contrast provides a route. Project the vector $v$ corresponding to a close pair onto $\mathcal N(\Phi_Q^\top)$ and normalize the projection, as described above. Under suitable conditioning, the square of the resulting feasible unit contrast has conditional bias bounded by $CH_n^{2\beta_b}\varepsilon_n^2$. The ideal minimizer $\mathbf a_Q^\star$ has no larger worst-case conditional bias. 
    Since one can also show that there are sufficiently many cells containing such close-pair witnesses, the variance of the resulting average is suitably controlled.
\end{itemize}

\subsection{An Ideal but Numerically Infeasible Estimator}\label{sec:estimators}

We now elaborate on the preceding intuition and describe the estimators mathematically. As mentioned above, we first focus on an idealized estimator that may be difficult to compute in practice and then derive a practical estimator from it.

First, we need some notation and conventions. Fix $\lambda_0>0$ and let
\begin{equation}\label{eq:lr-partition-scales}
    r_n=\max\left\{0,\operatorname{round}\left(\log_2\frac{n}{\lambda_0}\right)\right\},
    \qquad
    J_n=2^{r_n},
    \qquad
    H_n=J_n^{-1/d}.
\end{equation}
Next, choose nonnegative integers $s_{n,1},\ldots,s_{n,d}$ whose sum is $r_n$ and which differ from one another by at most one. Starting with $[0,1]^d$, split coordinate $j$ into $2^{s_{n,j}}$ equal intervals. The resulting partition $\cQ_n$ consists of $J_n$ axis-aligned rectangles, or cells, 
of equal volume $J_n^{-1}$, corresponding to the ones from \Cref{sec:intuition}. Cell boundaries are assigned by a fixed half-open convention. Every side length lies between $\kappa_HH_n$ and $\kappa_H'H_n$ for constants $0<\kappa_H<\kappa_H'$ independent of $n$. In particular, $J_n\asymp n$, $H_n\asymp n^{-1/d}$, and $\operatorname{diam}(Q)\leq C_HH_n$ for every $Q\in\cQ_n$ and all sufficiently large $n$.

Write the closure of each cell as $\overline Q=b_Q+A_Q[0,1]^d$, where $A_Q$ is diagonal. For $X_i\in Q$, let $U_i=A_Q^{-1}(X_i-b_Q)\in[0,1]^d$ be its normalized coordinate.
We are now ready to build on the final intuition from \Cref{sec:intuition}.

Following the same line of logic, let $I_Q=\{i:X_i\in Q\}$ and $N_Q=|I_Q|$. Write the indices as $i_1<\cdots<i_{N_Q}$ and let $U_{i_r}$, $r=1,\ldots,N_Q$, denote the corresponding normalized covariates. Strictly speaking, the indices $i_r$ depend on $Q$; we suppress that dependence when it is clear from context. Define
\begin{equation*}
    \Phi_Q=
    \begin{pmatrix}
    \psi(U_{i_1})^\top\\
    \vdots\\
    \psi(U_{i_{N_Q}})^\top
    \end{pmatrix}
    \in\R^{N_Q\times q},
    \qquad
    Y_Q=(Y_{i_1},\ldots,Y_{i_{N_Q}})^\top,
    \qquad
    f_Q=(f(X_{i_1}),\ldots,f(X_{i_{N_Q}}))^\top.
\end{equation*}
For a contrast $\mathbf a\in\R^{N_Q}$, define the worst-case squared regression contrast 
\begin{equation}\label{eq:ideal-cell-objective}
    \mathcal B_Q(\mathbf a)
    =\sup_{f\in\cH^{\beta_b}(L;U)}
    \left(\sum_{r=1}^{N_Q}a_rf(X_{i_r})\right)^2.
\end{equation}
For a unit contrast, \eqref{eq:ideal-cell-objective} equals the worst-case conditional bias of $(\mathbf a^\top Y_Q)^2$. The idealized problem of minimizing this bias subject to polynomial annihilation is
\begin{equation}\label{eq:ideal-cell-optimization}
    \inf\left\{
    \mathcal B_Q(\mathbf a):
    \|\mathbf a\|_2=1,
    \ \Phi_Q^\top\mathbf a=0
    \right\}.
\end{equation}
Suppose for the moment that a vector $\mathbf a_Q^\star$ attains this infimum. We will argue that, for sufficiently well-behaved cells, $\mathbf a_Q^\star$ serves our purpose. The first consideration is bias, and this is where one must be most careful. As argued above, a naive analysis using a generic annihilator is suboptimal. To show that $\mathbf a_Q^\star$ gains in bias, we produce a feasible witness for the optimization problem, establish an improved bias bound for that witness, and then invoke the optimality of $\mathbf a_Q^\star$.

Again, one can 
gain 
intuition from the comparison between the close-pair estimator and the degree-$\ell$ polynomial-annihilation estimator based on points in $Q$. Consider a cell $Q$ containing precisely $q+2$ observations: two observations provide the close pair, and the remaining $q$ observations provide the degrees of freedom needed to correct the $q$ polynomial moments. Suppose two 
indices $i_{r_1}$ and $i_{r_2}$ satisfy
\begin{equation*}
    \|U_{i_{r_1}}-U_{i_{r_2}}\|_\infty\leq\varepsilon_n.
\end{equation*}
Their separation on the scale of the actual covariates $X_{i_{r_1}},X_{i_{r_2}}$ is then $O(H_n\varepsilon_n)$. Let $e=(i_{r_1},i_{r_2})$ and $d_e=e_{r_1}-e_{r_2}\in\R^{q+2}$. The vector closest to $d_e$ under the constraint $\Phi_Q^\top\mathbf a=0$ is the orthogonal projection of $d_e$ onto $\mathcal N(\Phi_Q^\top)$. More explicitly, on the event when $\Phi_Q^\top\Phi_Q$ is invertible,  write
\begin{equation}\label{eq:cell-projections}
    G_Q=\Phi_Q^\top\Phi_Q,
    \qquad
    P_Q=\Phi_QG_Q^{-1}\Phi_Q^\top,
    \qquad
    R_Q=I_{N_Q}-P_Q,
    \qquad
    v_e=d_e^\top R_Qd_e.
\end{equation}
Then $R_Qd_e$ solves the distance
minimization problem, and, whenever $v_e>0$, the normalized vector $R_Qd_e/\sqrt{v_e}$ is feasible for \eqref{eq:ideal-cell-optimization}. Most importantly, under additional regularity conditions on $Q$ that hold for sufficiently many cells with high probability, the absolute bias (conditional on   $X$) of this normalized witness satisfies
\begin{equation*}
    \mathcal B_Q\left(\frac{R_Qd_e}{\sqrt{v_e}}\right)
    \leq CH_n^{2\beta_b}\varepsilon_n^2,
\end{equation*}
which improves on the naive $H_n^{2\beta_b}$ bound for generic annihilators. Since there are $\Theta(n)$ cells and we search for a close pair in each one, we expect order $n\varepsilon_n^d$ such cells. Averaging over the good cells then gives variance of order $(n\varepsilon_n^d)^{-1}$. The desired mean squared error follows by balancing squared bias $H_n^{4\beta_b}\varepsilon_n^4$ and variance $(n\varepsilon_n^d)^{-1}$. Since $H_n\asymp n^{-1/d}$, this balance is achieved by
\begin{equation}\label{eq:lr-optimal-bandwidth}
    a_\star=\frac{d-4\beta_b}{d(d+4)},
    \qquad
    \varepsilon_n\asymp n^{-a_\star}.
\end{equation}

Although this intuition justifies constructing estimators based on $\mathbf{a}_Q^\star$, the optimization problem
in \eqref{eq:ideal-cell-optimization} may not be practical. 

The preceding discussion also reveals an estimator based only on empirically observable quantities; this is the core estimator developed below.
For completeness,
we first complete the analysis of the computationally infeasible estimator based on $\mathbf a_Q^\star$. To do so, consider cells on which the optimized contrast $\mathbf{a}_{Q}^\star$ improves on the generic $H_n^{2\beta_b}$ bound by a factor of order $\varepsilon_n^2$. Specifically, for $C>0$, define
\begin{equation}\label{eq:ideal-good-cells}
    \mathcal S_n(\varepsilon_n,C)
    =\left\{
    Q\in\cQ_n:
    N_Q=q+2,
    \ \mathcal B_Q(\mathbf a_Q^\star)
    \leq CH_n^{2\beta_b}\varepsilon_n^2
    \right\}.
\end{equation}
We will show that, for a sufficiently large constant $C$, $|\mathcal S_n(\varepsilon_n,C)|\gtrsim n\varepsilon_n^d$ with high probability. We first justify the existence of $\mathbf a_Q^\star$.

\begin{lemma}\label{lem:existence_aqstar}
On the event $N_Q\geq q+1$, the minimum in \eqref{eq:ideal-cell-optimization} is attained. Thus there exists
\begin{equation*}
    \mathbf a_Q^\star\in
    \operatorname*{argmin}\left\{
    \mathcal B_Q(\mathbf a):
    \|\mathbf a\|_2=1,
    \ \Phi_Q^\top\mathbf a=0
    \right\}.
\end{equation*}
\end{lemma}

The next lemma contains the crux of the argument and asserts that there are sufficiently many cells on which the optimized contrast has the desired bias. Its proof uses the witness construction above.

\begin{lemma}\label{lem:enough_good_Q}
Under conditions (i)--(iii), with $\beta_b>1$ and $d>4\beta_b$, there is $C>0$ such that, for every sequence 
$(\varepsilon_n)_{n\ge 1}$ such that
$\varepsilon_n\to0$ and $n\varepsilon_n^d\to\infty$, there exist constants $c_0,c_1,c_2>0$ for which
\begin{equation}\label{eq:enough-ideal-good-cells}
    \Pp\left\{
    |\mathcal S_n(\varepsilon_n,C)|<c_0n\varepsilon_n^d
    \right\}
    \leq c_1\exp(-c_2n\varepsilon_n^d)
\end{equation}
for all sufficiently large $n$.
\end{lemma}

Given the lemma, define $\mathcal G_n(\varepsilon_n)=\{|\mathcal S_n(\varepsilon_n,C)|\geq c_0n\varepsilon_n^d\}$ and the 
ideal estimator
\begin{equation}\label{eq:ideal-lr-estimator}
    \widehat\sigma_{n,\mathrm{LR}}^2
    =
    \begin{cases}
    \displaystyle
    \frac{1}{|\mathcal S_n(\varepsilon_n,C)|}
    \sum_{Q\in\mathcal S_n(\varepsilon_n,C)}
    (\mathbf a_Q^{\star\top}Y_Q)^2,
    &\mathcal G_n(\varepsilon_n),\\[3mm]
    (\underline v+\overline v)/2,
    &\mathcal G_n(\varepsilon_n)^c.
    \end{cases}
\end{equation}
The fallback uses the bounds $\underline v$ and $\overline v$, which are part of the model definition.
This estimator satisfies the desired upper bound.

\begin{proposition}[Ideal estimator]\label{prop:ideal_est_lr}
Let $\varepsilon_n$ be chosen as in \eqref{eq:lr-optimal-bandwidth}. Under conditions (i)--(iii), with $\beta_b>1$ and $d>4\beta_b$,
\begin{equation}\label{eq:ideal-estimator-rate}
    \sup_{\theta\in\Theta(\beta_b,\beta_g,d)}
    \E_\theta(\widehat\sigma_{n,\mathrm{LR}}^2-\sigma^2)^2
    \leq Cn^{-4(\beta_b+1)/(d+4)}.
\end{equation}
\end{proposition}
We next study a numerically feasible estimator.

\subsection{A Numerically Feasible Estimator}\label{sec:est_num_feas}

The intuition behind the success of $\widehat\sigma_{n,\mathrm{LR}}^2$ motivates a fully implementable estimator that uses only the witnesses appearing in the proof of \Cref{lem:enough_good_Q}.

First consider any cell $Q$ with $N_Q=q+2$ and order its elements by increasing sample indices. Fix constants $\kappa>0$ and $v_0>0$, and form $G_Q=\Phi_Q^\top\Phi_Q$. Reject the cell unless $\lambda_{\min}(G_Q/N_Q)\geq\kappa$; on a cell passing this screen, define $P_Q$ and $R_Q$ as in \eqref{eq:cell-projections}. For an unordered pair $e=(i,j)\subset I_Q$, let $d_e=e_r-e_s\in\R^{N_Q}$ when $i=i_r$ and $j=i_s$, and define
\begin{equation}\label{eq:eligible-cell-pairs}
    \mathcal E_Q(\varepsilon_n)
    =\left\{
    (i,j)\subset I_Q:
    \|U_i-U_j\|_\infty\leq\varepsilon_n,
    \ d_e^\top R_Qd_e\geq v_0
    \right\}.
\end{equation}
We restrict to cells $Q$ with $\mathcal E_Q(\varepsilon_n)\neq\varnothing$. As discussed above, we seek an annihilator that mimics the close-pair contrast. To obtain the required uniform bias control, we additionally require 
that $G_Q$ is well conditioned.
Specifically, call $Q$ a good $(\varepsilon_n,\kappa)$-witness if
\begin{equation}\label{eq:feasible-cell-screen}
    N_Q=q+2,
    \qquad
    \lambda_{\min}(G_Q/N_Q)\geq\kappa,
    \qquad
    \mathcal E_Q(\varepsilon_n)\neq\varnothing.
\end{equation}

For each good witness cell, choose the lexicographically first pair $e_Q\in\mathcal E_Q(\varepsilon_n)$. Let $v=d_{e_Q}$ and define the normalized projected contrast
\begin{equation}\label{eq:feasible-witness-contrast}
    \widetilde{\mathbf a}_Q
    =\frac{R_Qd_{e_Q}}
    {\sqrt{d_{e_Q}^\top R_Qd_{e_Q}}}.
\end{equation}
Let $\mathcal T_n(\varepsilon_n)$ be the collection of good $(\varepsilon_n,\kappa)$-witness cells and let $\widetilde{\mathcal G}_n(\varepsilon_n)=\{|\mathcal T_n(\varepsilon_n)|\geq c_0n\varepsilon_n^d\}$. Define
\begin{equation}\label{eq:feasible-lr-estimator}
    \widetilde\sigma_{n,\mathrm{LR}}^2
    =
    \begin{cases}
    \displaystyle
    \frac{1}{|\mathcal T_n(\varepsilon_n)|}
    \sum_{Q\in\mathcal T_n(\varepsilon_n)}
    (\widetilde{\mathbf a}_Q^\top Y_Q)^2,
    &\widetilde{\mathcal G}_n(\varepsilon_n),\\[3mm]
    (\underline v+\overline v)/2,
    &\widetilde{\mathcal G}_n(\varepsilon_n)^c.
    \end{cases}
\end{equation}

\begin{proposition}[Feasible witness estimator]\label{prop:ideal_est_lr2}
Under conditions (i)--(iii), with $\beta_b>1$ and $d>4\beta_b$, the constants $\kappa$, $v_0$, and $c_0$ can be chosen so that, with $\varepsilon_n$ as in \eqref{eq:lr-optimal-bandwidth},
\begin{equation}\label{eq:feasible-estimator-rate}
    \sup_{\theta\in\Theta(\beta_b,\beta_g,d)}
    \E_\theta(\widetilde\sigma_{n,\mathrm{LR}}^2-\sigma^2)^2
    \leq Cn^{-4(\beta_b+1)/(d+4)}.
\end{equation}
\end{proposition}
Before proceeding, we present the core technical argument: bias reduction in good witness cells and the availability of sufficiently many such cells. These facts underlie the proof of \Cref{prop:ideal_est_lr2}. 

\subsection{Core technical ingredients}\label{sec:core-technical}

The ideal and feasible witness estimators rely on two facts: a projected close pair has a reduced conditional bias, and sufficiently many cells contain an eligible projected pair.

The first result is the deterministic reduction in the bias bound.
\begin{lemma}[Projected close-pair bias bound]\label{lem:pe-bias bound}
There is $C<\infty$ such that the following holds for every $t\in(0,1]$. On every cell satisfying $\lambda_{\min}(G_Q/N_Q)\geq\kappa$ and every pair $e=(i,j)$ satisfying $\|U_i-U_j\|_\infty\leq t$ and $v_e=d_e^\top R_Qd_e\geq v_0$,
\begin{equation}\label{eq:pe-bias bound-bound}
\left|
\frac{d_e^\top R_Qf_Q}{\sqrt{v_e}}
\right|
\leq CH_n^{\beta_b}t.
\end{equation}
\end{lemma}
For the feasible witness estimator, take $t=\varepsilon_n$ above.
Next, 
to argue that sufficiently many good cells are available,
let $\mathcal C_n(h)$ be the collection of cells $Q\in\cQ_n$ satisfying
\begin{equation}\label{eq:certificate-cell-definition}
N_Q=q+2,
\qquad
\lambda_{\min}(G_Q/N_Q)\geq\kappa,
\qquad
\mathcal E_Q(h)\neq\varnothing,
\end{equation}
where $\mathcal E_Q(h)$ is defined as in \eqref{eq:eligible-cell-pairs} with $h$ in place of $\varepsilon_n$. We have the following.

\begin{lemma}[Availability of projected pairs]\label{lem:pe-availability}
The constants $\kappa$ and $v_0$ can be chosen so that, whenever $t_n\to0$ and $nt_n^d\to\infty$,
\begin{equation}\label{eq:pe-availability}
\Pp\left\{
|\mathcal C_n(t_n)|<cnt_n^d
\right\}
\leq C\exp(-cnt_n^d)
\end{equation}
for all sufficiently large $n$ and constants $c,C>0$.
\end{lemma}
 Although the result above demonstrates exponentially decaying probability of having fewer than a fixed positive multiple of $nt_n^d$ cells, for the proofs of \Cref{prop:ideal_est_lr,prop:ideal_est_lr2} a weaker result with a simpler proof suffices. We present it in Section \ref{sec:availability-chebyshev} and use it in the risk proofs below, while also recording the alternative based on the exponential bound.

\section{Related work}\label{sec:related}

{\bf Difference-based, residual-based, and matched-pair estimators.}
Difference-based variance estimators have a long history; see, among others,
\citet{von1941mean},
\citet{Rice1984}, \citet{HallKayTitterington1990}, \citet{BrownLevine2007}, and \citet{WangEtAl2008}. Under a fixed design, smoothness of the regression function controls the difficulty of variance estimation; see \citet{MunkEtAl2005} and \citet{CaiLevineWang2009}. These procedures use finite differences chosen to annihilate low-order polynomial trends along a deterministic design.

Residual-based procedures form a second classical family. \citet{HallMarron1990} estimate variance from squared nonparametric regression residuals with bias or degrees-of-freedom correction, and \citet{LiLin2020} develop a semiparametrically efficient residual-based construction. 

A third family uses matched covariates or pairwise extrapolation. \citet{MullerSchickWefelmeyer2003} proposed a covariate-matched $U$-statistic whose weights use a kernel estimate of the design density, and \citet{DuSchick2009} developed a related covariate-matched estimator. Nearest-neighbor graph constructions were studied by \citet{LiitiainenEtAl2010}. \citet{TongWang2005} regress squared response differences on squared covariate distances and estimate the variance by extrapolating to zero distance; \citet{TongMaWang2013} give a further difference-based optimality analysis. \citet{ShenEtAl2020} establish sharp random-design results for a kernel-weighted pairwise-difference ratio and related local polynomial variance-function estimators. Their results settle the one-dimensional minimax rate, while the multivariate problem remains more delicate.

The appendix extension in \Cref{sec:est_num_stable} combines these ideas differently. Like the core  estimator, it first imposes exact polynomial annihilation on the realized design by an empirical projection. It then reuses all eligible fine-scale pair contrasts and performs a stable distance extrapolation.

{\bf Random design and smoothness above one.}
\citet{RobinsEtAl2008} obtained the rate $n^{-8\beta_b/(d+4\beta_b)}$ in the low-smoothness regime $0<\beta_b<1$. \citet{ShenEtAl2020} developed sharp random-design procedures based on pairwise differences and local polynomial smoothing; their multivariate homoskedastic upper bound is proved when every coordinatewise smoothness exponent is at most one. \citet{AronowLopatto2026} show that the Robins rate is not uniformly attainable for $\beta_b>1$ and $d>4\beta_b$ when the design density is subject only to upper and lower bounds. Their result does not by itself establish the same impossibility on each fixed H\"older-density subclass considered here. 

{\bf Higher-order influence functions and smooth design densities.}
The expected conditional covariance framework of \citet{RobinsEtAl2008} includes the present target as a special case. Taking $A=Y$, setting the two conditional mean nuisances equal to $f$, and taking $H_1=1$, $H_2=H_3=-Y$, and $H_4=Y^2$ makes their target equal to $\E\{\operatorname{var}(Y\mid X)\}=\sigma^2$. Under this specialization, their smoothness condition reduces to $\beta_g>\beta_g^\star$, with $\beta_g^\star$ given in \eqref{eq:robins-density-threshold}. Their finite-order estimator supplies the comparison underlying the first branch of \Cref{thm:main}, subject to the additional uniform moment condition \eqref{eq:hoif-uniform-moment-condition}. 

{\bf Double-cross-fit doubly robust estimation.}
A double-cross-fit doubly robust estimator for expected conditional covariance, introduced by \citet{newey2018cross} and studied in a rough non-$\sqrt n$ regime by \citet{mcclean2026double}, can also be specialized by taking both variables to be $Y$. 
Its existing analyses are formulated under a different loss and nuisance parameter setup from our paper, but we compare with it in our experiments nonetheless.

\section{Numerical experiments}\label{sec:numerical-rate}

\subsection{Compared estimators}

We compare the two estimators discussed in the paper with five baselines, for seven methods in total.
\begin{enumerate}[label=(\roman*),leftmargin=2.4em]
\item \emph{Feasible witness estimator.} 
This is the estimator from \eqref{eq:feasible-lr-estimator}: among cells containing $q+2$ observations, it retains the lexicographically first eligible projected pair, normalizes its contrast, and averages one squared contrast per accepted cell with equal weights. 
\item \emph{Two-scale estimator.} 
This is the extension from \Cref{sec:est_num_stable}. It  reuses all eligible projected pairs, and estimates the zero-distance intercept using ridge weights before clipping.
\item \emph{Newey--Robins DCDR.} We specialize the double-cross-fit doubly robust expected conditional covariance estimator of \citet{newey2018cross,mcclean2026double} to $A=Y$. We split the observations into three folds by sample label and cycle the folds. 
For each evaluation fold, two separate cellwise local-affine least-squares regressions are trained on the other two folds, and we average the products of the two residuals over the evaluation observations.
\item \emph{Close pairs.} We form disjoint consecutive pairs within cells, retain pairs whose normalized $\ell_\infty$ distance is bounded by a specified level, and average $(Y_i-Y_j)^2/2$.
\item \emph{All-pairs kernel $U$-statistic.} Motivated by \citet{MullerSchickWefelmeyer2003,DuSchick2009,ShenEtAl2020}, we average $(Y_i-Y_j)^2/2$ over all same-cell pairs whose normalized $\ell_\infty$ separation is at most a given value.
\item \emph{Tong--Wang extrapolation.} Following \citet{TongWang2005}, we form all same-cell pairs with normalized $\ell_\infty$ separation at most $0.5$, regress $(Y_i-Y_j)^2/2$ on $\|X_i-X_j\|_2^2$, and use the fitted intercept. 

\item \emph{Cellwise local-affine residuals.} In each cell containing more than $d+1$ observations, we fit an affine regression by ordinary least squares, pool the residual sums of squares, and divide by the total residual degrees of freedom. 
\end{enumerate}

\subsection{Hyperparameter selection}\label{subsec:expanded-tuning}

Throughout this section, write $\beta=\beta_b$. We perform a pilot study to select hyperparameters for all methods and use the selected values in all three experiments. The feasible witness estimator and the two-scale estimator are tuned separately.
For the pilot, we use
$d\in\{6,8,10\}$ and $\beta\in\{1.1,1.5,2\}$ with
a variety of designs, regression functions, and 
signal-to-noise ratios. 
We select hyperparameters in a way that
balances average relative performance and upper-tail performance.
See \Cref{tab:expanded-defaults}
for the searched values and selected defaults; the details are available in the associated code repository.

For the feasible witness estimator, the pilot calibration selects target cell occupancy $8$, bandwidth $\varepsilon_n=\min\{1,0.7n^{-a_\star}\}$, Gram threshold $\kappa=0.001$, and contrast-variance threshold $v_0=0.4$; the good-event fraction is fixed at $c_0=0.01$. Target occupancy controls the partition resolution, while retained witness cells continue to have exactly $q+2$ observations as required by \eqref{eq:feasible-cell-screen}. The separately selected defaults for the stable two-scale extension are reported in \Cref{tab:expanded-defaults}.

\begin{table}
\centering
\small
\begin{tabular}{lll}
\toprule
Hyperparameter & Values represented in the search & Selected default \\
\midrule
Target cell occupancy & $\{64,96,128,192,256\}$ & 64 \\
Close-pair constant $c_\rho$ & $\{0.30,0.45,0.60,0.80,1.00\}$ & 0.8 \\
Bandwidth multiplier & $\{1,1.5,2,2.5,3\}$ & 1.5 \\
Maximum normalized radius & $\{0.35,0.50,0.70,0.90\}$ & 0.5 \\
Gram threshold $\kappa$ & $\{0.01,0.03,0.05,0.08,0.12\}$ & 0.05 \\
Maximum cell occupancy $N_{\max}$ & occupancy $\times\{1.5,2,3,4\}$ & 192 \\
Contrast-variance threshold $v_0$ & $\{0.25,0.40,0.50,0.65,0.80\}$ & 0.8 \\
Distance ridge $\lambda$ & $\{0.001,0.003,0.01,0.03,0.10\}$ & 0.003 \\
\bottomrule
\end{tabular}
\caption{Selected hyperparameters for the stable all-pairs two-scale estimator.}
\label{tab:expanded-defaults}
\end{table}

Similarly,
 the hyperparameters of 
each baseline are tuned on the same training designs and then frozen.  Close pairs use a $5\times5$ grid over occupancy and $c_\rho$; the kernel estimator uses a $5\times5$ grid over occupancy and bandwidth constant; Tong--Wang uses a $5\times4$ grid over occupancy and normalized radius; and local polynomial residuals and DCDR use seven occupancy values each.
The selected settings are in \Cref{tab:baseline-defaults}. These numerical thresholds are empirical choices; the rate guarantees apply to the admissible fixed constants constructed in the proofs, not to every possible tuning choice.

\begin{table}
\centering
\small
\begin{tabular}{lrrl}
\toprule
Comparator & Target occupancy & Method-specific parameter & Training score \\
\midrule
Close pairs & 64 & $c_\rho=0.8$ & 3.37 \\
Kernel $U$-statistic & 64 & $c_h=0.5$ & 1.99 \\
Tong--Wang & 192 & radius $=0.5$ & 3.83 \\
Local polynomial & 64 & -- & 1.07 \\
DCDR & 192 & -- & 1.24 \\
\bottomrule
\end{tabular}
\caption{Hyperparameters selected for the baselines.}
\label{tab:baseline-defaults}
\end{table}

\subsection{Monte Carlo protocol}

In the two simulation experiments, every design replication samples $n$ iid covariates; each empirical design replication instead draws a subsample without replacement from the 
full dataset. Apart from deterministic fallbacks used when screening leaves too few eligible objects, conditional on the sampled covariates each quadratic estimator can be written as $Y^\top A_XY$ with $A_X=A_X^\top$ and $\operatorname{tr}(A_X)=1$. Writing $f=(f(X_1),\ldots,f(X_n))^\top$, under iid Gaussian errors its exact conditional mean squared error is
\begin{equation}\label{eq:quadratic-risk-simulation}
\E\left[\left\{Y^\top A_XY-\sigma^2\right\}^2\mathrel{\big|}X\right]
=
\left(f^\top A_Xf\right)^2
+2\sigma^4\operatorname{tr}(A_X^2)
+4\sigma^2f^\top A_X^2f.
\end{equation}
For weighted pairwise estimators, 
we evaluate 
the trace and quadratic terms 
numerically, 
retaining all dependence induced by overlapping pairs. 
On the good event, the matrix of the feasible witness estimator is the direct sum of the equal-weight rank-one blocks $|\mathcal T_n(\varepsilon_n)|^{-1}\widetilde{\mathbf a}_Q\widetilde{\mathbf a}_Q^\top$ over disjoint witness cells. If its good event fails, we instead evaluate the exact squared error of the deterministic midpoint fallback in \eqref{eq:feasible-lr-estimator}. This estimator is left unclipped. 

For the two-scale estimator, $A_X$ is the matrix in \eqref{eq:pe-quadratic-form}. For DCDR, $A_X$ is the symmetrized quadratic form generated by the cyclic products of residual operators from the two separately trained nuisance fits. For the local-affine estimator, $A_X=R_X/D_X$, where $R_X$ is the block-diagonal residual projection and $D_X=\operatorname{tr}(R_X)$.

We evaluate \eqref{eq:quadratic-risk-simulation} and average it over independent designs or empirical subsamples. For the two-scale estimator, the reported exact risk is that of the quadratic estimator before clipping and therefore upper bounds the risk of its clipped implementation.
For continuity across all summaries, relative MSE is normalized within each setting by the MSE of the two-scale estimator.

\subsection{Evaluating the rate}

Our first experiment examines the finite-sample risk decay of both proposed estimators.
We use $d=8$ and $\beta=3/2$. 
Thus 
$\ell=1$, $q_{d,\ell}=9$, and the theoretical mean squared-error upper-bound exponent for both the feasible witness estimator in \Cref{prop:ideal_est_lr2} and the stable extension in \Cref{prop:stable-est-lr} is
$4(\beta+1)/(d+4)=5/6$.
The regular grid
exponent is $4\beta/d=3/4$, while the close pair exponent is $8/(d+4)=2/3$.
The design is uniform on $[0,1]^8$, the errors are iid $\mathcal{N}(0,10^{-2})$, and the regression functions form the triangular array over $n\ge1$,
$$
f_n(x)
=
3\left(v^\top x-\sqrt8/2\right)
+H_n^{3/2}\omega^{-3/2}
\cos\left(
\omega\left(v^\top x-\sqrt8/2\right)/H_n+0.37
\right),
$$
where $v=8^{-1/2}(1,\ldots,1)^\top$, $\omega=1.055$, and $H_n=J_n^{-1/8}$ is the common reference cell scale used to define the triangular array, rather than an estimator-specific scale.
The oscillatory amplitude keeps the $C^{1,1/2}$ norm bounded uniformly in $n$. The local affine component is relatively large, making close-pair differences biased.
Due to computational cost, 
the numbers of retained design replications are $100$, $60$, $18$, $5$, and $5$ at $n\in\{8192,32768,131072,524288,10^6\}$. 

\begin{figure}
\centering
\includegraphics[width=\textwidth]{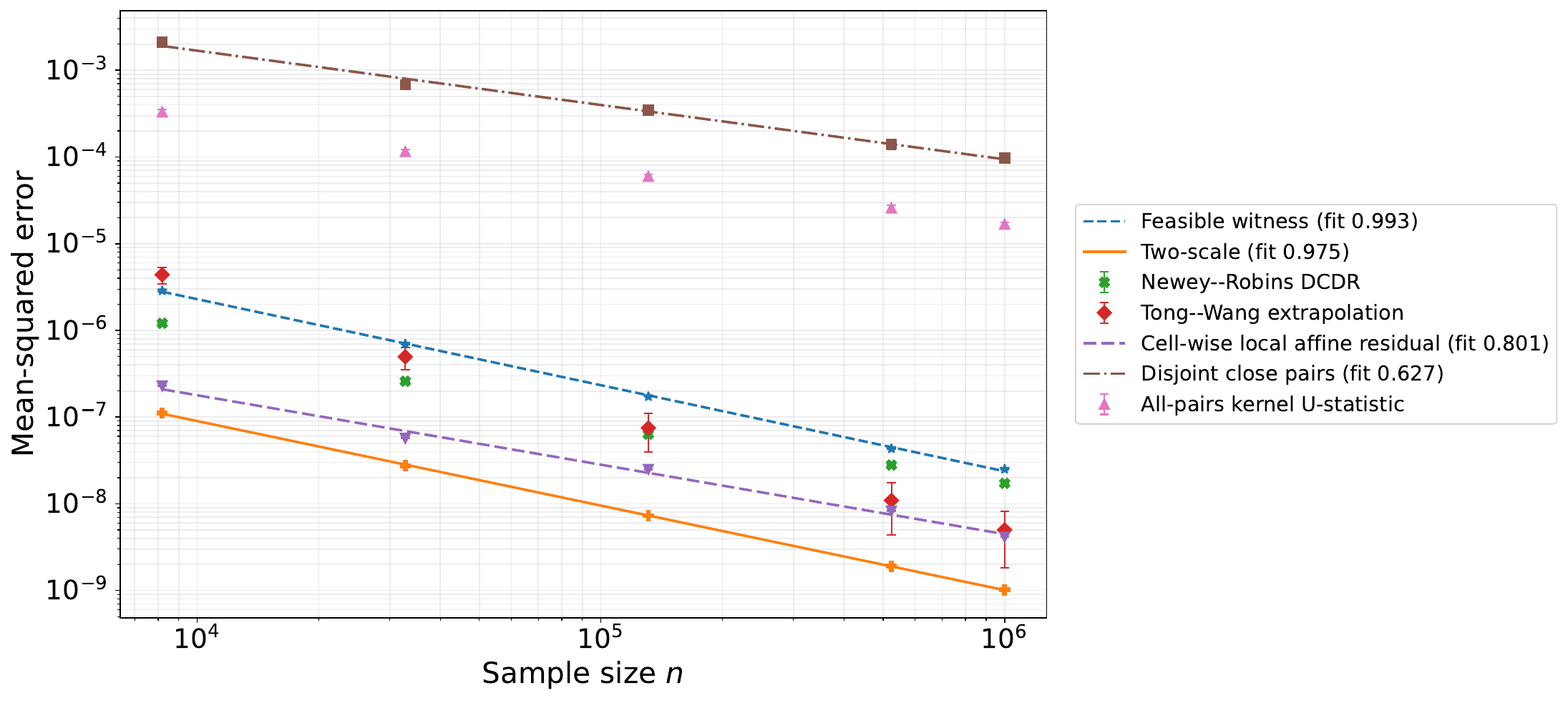}
\caption{Design-averaged exact conditional Gaussian mean squared error for the feasible witness estimator, the two-scale estimator, and five baselines, with the two-scale estimator evaluated before clipping and error bars equal to two design Monte Carlo standard errors. Ordinary least squares power-law fits to $\log(\mathrm{MSE})$ versus $\log n$ over all five sample sizes are shown for the feasible witness estimator, the two-scale estimator, the cellwise local-affine residual estimator, and disjoint close pairs. Their fitted exponents are $0.993$, $0.975$, $0.801$, and $0.627$, respectively.}
\label{fig:numerical-rate}
\end{figure}

The MSE of the feasible witness estimator decreases from $2.87\times10^{-6}$ at $n=8192$ to $2.51\times10^{-8}$ at $n=10^6$, with fitted exponent $0.993$. The MSE of the two-scale estimator decreases from $1.12\times10^{-7}$ to $1.02\times10^{-9}$ over the same range, with fitted exponent $0.975$.
The Newey--Robins DCDR MSE decreases from $1.21\times10^{-6}$ to $1.72\times10^{-8}$, with fitted exponent $0.873$. The cellwise local-affine residual MSE decreases from $2.28\times10^{-7}$ to $4.12\times10^{-9}$, with fitted exponent $0.801$. The fitted exponents for the proposed methods exceed $5/6$; this is compatible with a worst-case upper bound and does not establish that the bound is sharp for this uniform-design triangular array. 

The feasible witness estimator uses an average of $72.65$ disjoint witness cells at $n=8192$ and $8{,}045.8$ at $n=10^6$, with no fallback in any retained rate-experiment replication. At $n=10^6$, the two-scale estimator uses an average of approximately $3.05$ million close residual pairs. After accounting exactly for their dependence, 
the mean effective rank 
of its quadratic form 
is approximately $2.36\times10^5$. 
The fixed ridge keeps the intercept weights stable: over the retained sizes, the design-averaged $\ell_1$ norm is about $2.42$, and the design-averaged value of $M_n\max_e|w_e|$ is at most about $14.8$.


\subsection{Varying the feature density and regression function}\label{subsec:expanded-gaussian}

Next, 
we report experiments with a variety of feature densities and regression functions. 
We evaluate four design distributions with support $[0,1]^d$ and densities bounded above and away from zero for fixed $d$.  Besides the uniform density, we use
\begin{align*}
p_{\mathrm{tilt}}(x)
&=1+0.6(2x_1-1),\qquad
p_{\mathrm{check}}(x)
=1+0.55(-1)^{\sum_{j=1}^{\min(3,d)}\lfloor2x_j\rfloor},\\
p_{\mathrm{cluster}}(x)
&=0.35+\frac{0.65}{2}\left\{\prod_{j=1}^d b_{2,5}(x_j)+\prod_{j=1}^d b_{5,2}(x_j)\right\},
\end{align*}
where $b_{a,b}$ is the beta density.  The tilted design tests a smooth nonuniform marginal, the checkerboard design has a discontinuous density, and the clustered mixture has two dependent latent clusters. The checkerboard density is outside every positive-smoothness H\"older class in condition (i), but is covered by the density-bounded extension of the two-scale results noted after \Cref{thm:main}.

\paragraph{Regression functions.}
Let $v=d^{-1/2}(1,\ldots,1)^\top$, $t=v^\top x-\sqrt d/2$, $z_j=x_j-1/2$, $r=\min(4,d)$, and $H=J_n^{-1/d}$ be the cell scale.  
We evaluate a variety of globally smooth, irregular, and cell-scale regression functions:
\begin{align*}
f_{\mathrm{aff}}(x)&=1.8t,\qquad
f_{\mathrm{add}}(x)=d^{-1/2}\sum_{j=1}^d\sin(2\pi\nu_jx_j),\qquad \nu_j=1+((j-1)\bmod3),\\
f_{\mathrm{ridge}}(x)&=1.2\sin(2\pi t/\sqrt d)+0.6t^2,\qquad
f_{\mathrm{int}}(x)=\frac{4}{r-1}\sum_{j=1}^{r-1}z_jz_{j+1}+0.5\sin(2\pi t),\\
f_{\mathrm{cusp}}(x)&=2\operatorname{sign}(t)|t|^\beta+0.45t,\qquad
f_{\mathrm{osc}}(x)=\sin(2\pi c x_1)+0.55\cos(2\pi c x_2),\\
f_{\mathrm{tri},H}(x)&=2.5t+H^\beta\omega^{-\beta}\cos(\omega t/H+0.37),
\end{align*}
where $c=4$ for $d\leq8$, $c=3$ for $d=10$, and $\omega=1.055$.  These functions are rescaled on the realized design so that $n^{-1}\sum_i f(X_i)^2/\sigma^2$ equals the stated SNR. This design-dependent rescaling is used for finite-sample comparison: \eqref{eq:quadratic-risk-simulation} remains exact conditional on the design, but the resulting ensemble is not an iid experiment with a single fixed regression function.

We also include two regression functions that
place nonpolynomial variation at the cell scale:
\begin{align*}
f_{\mathrm{stress},1,H}(x)
&=3t+1.75H^\beta\omega^{-\beta}\cos(\omega t/H+0.37),\\
f_{\mathrm{stress},2,H}(x)
&=1.8t+\frac{1.35H^\beta}{\sqrt r}
\sum_{j=1}^r\omega_j^{-\beta}\cos(\omega_jz_j/H+\phi_j),
\end{align*}
where $\omega_j=1+0.15(j-1)$ and $\phi_j=0.31+0.47(j-1)$.  
These two stress functions are not rescaled to a fixed SNR, because doing so would obscure the intended uniform H\"older scaling of the cell-scale perturbation.

\begin{table}
\centering
\scriptsize
\begin{tabular}{lrrllll}
\toprule
ID & $d$ & $\beta$ & Design & Mean & SNR & $\sigma^2$ \\
\midrule
S1 & 6 & 1.1 & Uniform & Affine & 4 & 0.01 \\
S2 & 6 & 1.1 & Tilted & Cusp & 4 & 0.01 \\
S3 & 6 & 1.1 & Checkerboard & Oscillatory & 8 & 0.01 \\
S4 & 8 & 1.5 & Uniform & Triangular & 4 & 0.01 \\
S5 & 8 & 1.5 & Tilted & Additive & 4 & 0.01 \\
S6 & 8 & 1.5 & Checkerboard & Cusp & 8 & 0.01 \\
S7 & 8 & 1.5 & Clustered & Interaction & 0.5 & 0.10 \\
S8 & 10 & 2.0 & Uniform & Ridge & 4 & 0.01 \\
S9 & 10 & 2.0 & Tilted & Triangular & 8 & 0.01 \\
S10 & 10 & 2.0 & Clustered & Cusp & 0.5 & 0.10 \\
H1 & 6 & 1.1 & Uniform & $f_{\mathrm{stress},1,H}$ & -- & 0.01 \\
H2 & 6 & 1.1 & Checkerboard & $f_{\mathrm{stress},2,H}$ & -- & 0.01 \\
H3 & 8 & 1.5 & Tilted & $f_{\mathrm{stress},1,H}$ & -- & 0.01 \\
H4 & 8 & 1.5 & Clustered & $f_{\mathrm{stress},2,H}$ & -- & 0.01 \\
H5 & 10 & 2.0 & Checkerboard & $f_{\mathrm{stress},1,H}$ & -- & 0.01 \\
H6 & 10 & 2.0 & Clustered & $f_{\mathrm{stress},2,H}$ & -- & 0.01 \\
\bottomrule
\end{tabular}
\caption{Simulation settings.}
\label{tab:expanded-scenarios}
\end{table}

\paragraph{Monte Carlo sizes and summary measure.}
For scenarios S1--S10, the numbers of independent design replications at $n=4096,8192,16384$ are $24$, $16$, and $10$; for scenarios H1--H6 they are $16$, $10$, and $6$.  This gives 692 design replications and 4,844 exact method risks across seven methods. 
For each method, we pool its 48 ratios into one comparison and report their geometric mean and the number of settings in which the method has the smallest MSE.

\begin{table}
\centering
\small
\begin{tabular}{lrr}
\toprule
Method & Geo. mean rel. MSE & Wins \\
\midrule
Feasible witness estimator & 30.81 & 1/48 \\
Two-scale estimator & 1.00 & 15/48 \\
Local polynomial & 1.08 & 11/48 \\
Tong--Wang & 5.21 & 4/48 \\
Kernel $U$-statistic & 36.94 & 5/48 \\
DCDR & 2.00 & 12/48 \\
Close pairs & 230.19 & 0/48 \\
\bottomrule
\end{tabular}
\caption{Aggregate Gaussian exact-risk results pooled across all 48 scenario--sample-size settings. Relative MSE is normalized by the two-scale estimator separately in each setting; values below one favor the displayed method over the two-scale estimator. A win denotes the smallest design-averaged MSE in a setting.}
\label{tab:expanded-simulation-summary}
\end{table}

\begin{figure}
\centering
\includegraphics[width=0.7\textwidth]{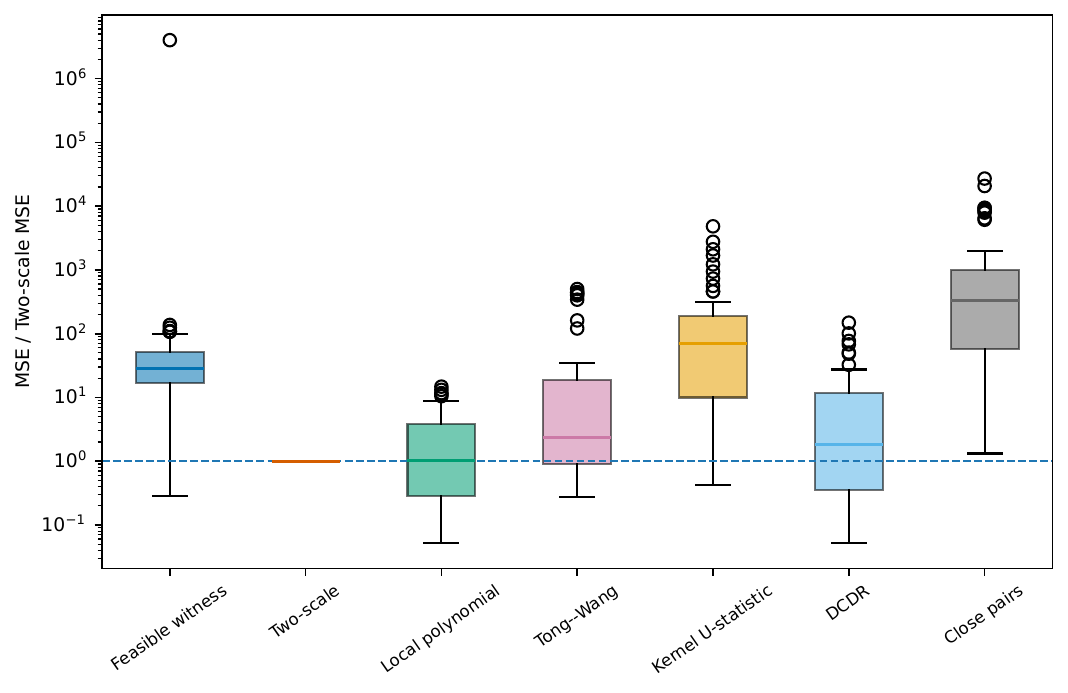}
\caption{Distribution of MSE ratios pooled across all 48 scenario--sample-size settings. The horizontal line at one is the two-scale estimator; the feasible witness estimator is displayed separately.}
\label{fig:expanded-simulation-boxplots}
\end{figure}

\begin{figure}
\centering
\includegraphics[width=\textwidth]{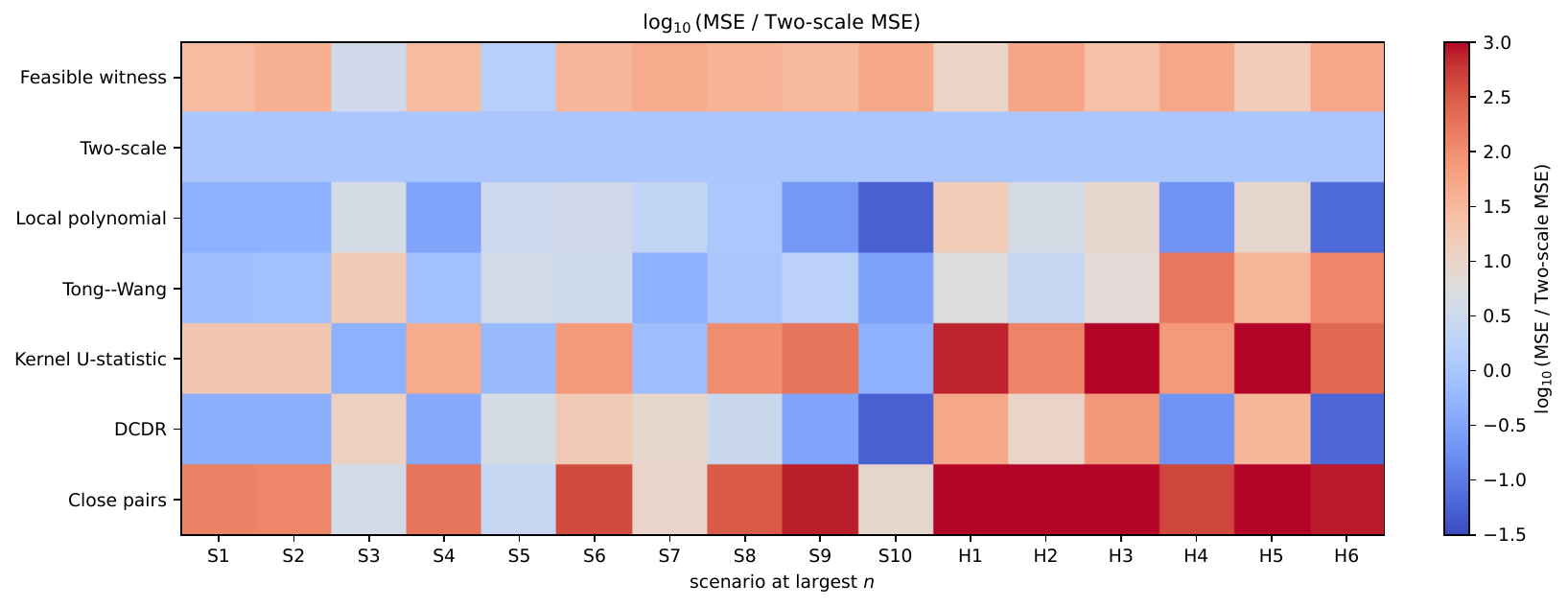}
\caption{Scenario-level MSE ratios at $n=16{,}384$, on a base-ten logarithmic scale. Blue entries favor the displayed method and red entries favor the two-scale estimator.}
\label{fig:expanded-simulation-heatmap}
\end{figure}

\paragraph{Results.}
The two-scale estimator has the smallest MSE in 15 of 48 settings, the largest win count among the seven methods. Its geometric mean relative MSE is the reference value one; the corresponding pooled ratios are $1.08$ for local polynomial residuals, $2.00$ for DCDR, $5.21$ for Tong--Wang, $30.81$ for the feasible witness estimator, $36.94$ for the kernel estimator, and $230.19$ for close pairs. 
The feasible witness estimator has median relative MSE $28.74$ and wins one setting; local polynomial residuals and DCDR win 11 and 12 settings, respectively. In one of the 692 design replications for the feasible witness estimator, fewer than the required number of witness cells remained and the estimator returned its deterministic midpoint fallback. We retain this outcome; it produces the method's long upper tail (and a maximum setting-level relative MSE of about $4.0\times10^6$).

The heat map clarifies the heterogeneity behind the aggregate: low-order polynomial or locally benign regression functions often favor simpler residual estimators, which avoid the variance cost of projection, screening, and extrapolation, whereas the two-scale estimator is strongest when cell-scale H\"older variation leaves a leading bias after local Taylor residualization.

\subsection{Experiments with empirical data}\label{subsec:expanded-empirical}

We illustrate both the feasible witness and the two-scale methods using two empirical datasets. 
Since the true response need not have constant variance, we consider a semi-synthetic protocol: the covariates are retained and a smooth regression function is learned from the data, but independent homoskedastic Gaussian errors with known variance are added for evaluation.  

\paragraph{Datasets.}
The California Housing data contain $20{,}640$ census-block-group observations and eight numeric predictors \citep{PaceBarry1997}.  After removing observations missing a required field, $20{,}433$ complete cases remain.  The CASP physicochemical protein tertiary structure data contain $45{,}730$ observations, nine predictors, and RMSD as the response \citep{Rana2013CASP}.


Our statistical model assumes a 
full-dimensional density on $[0,1]^d$. 
The estimator partitions the ambient cube, searches for close pairs within cells, and retains a cell only when its local polynomial Gram matrix is sufficiently well conditioned.  If several predictors are nearly linearly redundant, the observations concentrate near a lower-dimensional space, so many cells receive little mass. 

To reduce empirical near-collinearity before applying the methods, we use the same whitening pipeline for both data sets. 
A coordinatewise rank transform equalizes marginal scales and places every feature in $(0,1)$, but, being marginal, it leaves this joint dependence essentially unchanged.  For CASP, after this first transform the median and maximum absolute Spearman correlations are $0.844$ and $0.998$.

With a fixed seed for randomized tie-breaking, if $r_{ij}$ denotes the resulting rank, we set $U_{ij}=r_{ij}/(n+1)\in(0,1)$, after which we form Gaussian scores $Z_{ij}=\Phi_{\mathsf N}^{-1}(U_{ij})$, where $\Phi_{\mathsf N}$ is the standard normal distribution function, and center their columns.  Writing the empirical covariance as $\widehat\Sigma=V\operatorname{diag}(\widehat\lambda_1,\ldots,\widehat\lambda_d)V^\top$, we whiten by
\begin{equation*}
W=ZV\operatorname{diag}\bigl\{(\widehat\lambda_j\vee\tau)^{-1/2}\bigr\}_{j=1}^dV^\top,
\end{equation*}
where we set a small numerical floor $\tau=\max\{10^{-4},10^{-4}\widehat\lambda_{\max}\}$.  
The rank and Gaussian score steps give all coordinates a common marginal shape and scale, so the eigensystem of $\widehat\Sigma$ reflects joint linear dependence.  Whitening rescales each covariance eigen-direction to empirical variance $\widehat\lambda_j/(\widehat\lambda_j\vee\tau)$, which equals one when $\widehat\lambda_j\geq\tau$, thereby reducing linear anisotropy. This preprocessing is a conditioning device, not a verification of the full-dimensional density assumption; the experiments evaluate the methods on the transformed empirical designs.  

A second randomized marginal rank map returns each coordinate to $(0,1)$, and the resulting covariates are used by every method. 
The final median absolute Spearman correlations are $0.017$ for California and $0.035$ for CASP, while the corresponding maxima are $0.076$ and $0.328$.

For each dataset, the observed response is standardized and used only to fit a smooth pre-trained model. 
We use 384 random Fourier features for an RBF kernel, select the kernel scale from one-half, one, or twice the median-distance scale and ridge parameter from $\{10^{-3},10^{-2},10^{-1},1\}$ on a fixed holdout set, and then refit on all observations.  

\paragraph{Semi-synthetic regression functions.}
Let $g_i$ be the standardized fitted pre-trained model and let
\begin{equation*}
h_i=\frac{\operatorname{sign}(t_i)|t_i|^{3/2}-\overline h}
{\{n^{-1}\sum_{j=1}^n(\operatorname{sign}(t_j)|t_j|^{3/2}-\overline h)^2\}^{1/2}},
\qquad
t_i=d^{-1/2}\sum_{j=1}^dX_{ij}-\sqrt d/2.
\end{equation*}
Here $\overline h=n^{-1}\sum_{j=1}^n\operatorname{sign}(t_j)|t_j|^{3/2}$.
We evaluate a smooth data-informed regression function proportional to $g_i$ and a perturbed regression function proportional to $g_i+0.65h_i$.  Both are centered and rescaled to SNR $4$ with $\sigma^2=0.01$ and $\beta=1.5$.  For California, subsample sizes are $4096$, $8192$, and $16{,}384$, with $20$, $12$, and $8$ independent subsamples per regime. 
CASP additionally uses $n=32{,}768$, with $20$, $12$, $8$, and $6$ subsamples.  This gives fourteen
settings, 172 empirical subsamples, and 1,204 exact method risks across the seven methods.

\begin{table}
\centering
\small
\begin{tabular}{lrrrr}
\toprule
Method & Geo. mean rel. MSE & Median & 90th percentile & Wins \\
\midrule
Feasible witness estimator & 21.41 & 24.46 & 36.76 & 0/14 \\
Two-scale estimator & 1.00 & 1.00 & 1.00 & 14/14 \\
Local polynomial & 12.57 & 15.03 & 23.61 & 0/14 \\
Tong--Wang & 16.32 & 14.53 & 82.86 & 0/14 \\
Kernel $U$-statistic & 4.02 & 3.06 & 13.76 & 0/14 \\
DCDR & 51.55 & 55.29 & 117.66 & 0/14 \\
Close pairs & 29.29 & 21.15 & 83.44 & 0/14 \\
\bottomrule
\end{tabular}
\caption{Aggregate results over the fourteen empirical-design dataset--regime--sample-size settings. Relative MSE is normalized by the two-scale estimator within each setting.}
\label{tab:empirical-summary}
\end{table}


The two-scale estimator has the smallest design-averaged MSE in all fourteen settings. The kernel $U$-statistic is the closest aggregate comparator, with geometric mean relative MSE $4.02$; at $n=4096$ for the smooth CASP regression function its ratio is only $1.03$, but the gap grows with sample size. The geometric mean relative MSE is $12.57$ for local polynomial residuals, $16.32$ for Tong--Wang, $21.41$ for the feasible witness estimator, $29.29$ for close pairs, and $51.55$ for DCDR. The feasible witness estimator has median relative MSE $24.46$. The result for the two-scale estimator persists both for the smooth pilot and after adding the H\"older-boundary perturbation. 


\subsection{Overall interpretation}\label{subsec:expanded-discussion}

The experiments support that when the regression function is benign relative to the feature partition, local polynomial residuals or DCDR can be more efficient. The
feasible witness estimator shows decreasing risk in the rate experiment, but its use of only one pair per exact-occupancy cell entails a substantial finite-sample variance cost and occasional fallback risk. Reusing all eligible pairs and stabilizing the distance extrapolation can substantially reduce that cost: when cell-scale H\"older variation creates significant bias, the two-scale estimator provides the strongest overall performance in these experiments.  

In terms of limitations, 
hyperparameters are chosen only over $d\in\{6,8,10\}$ and $\beta\in\{1.1,1.5,2\}$; 
we expect higher polynomial degrees or substantially different dimensions may need a new hyperparameter sweep.


\section*{Discussion}
Although the results in this paper yield faster convergence rates than existing estimators in the literature, several directions remain open. In particular,
our work assumes that the degree of smoothness of the regression function is known. 
Data-driven adaptation to unknown smoothness and inference for the variance estimate remain separate problems.
 Moreover, the precise determination of the minimax rate also remains open.


 \section*{Acknowledgements}
 We are grateful to P.M. Aronow and Patrick Lopatto for helpful discussions and encouragement, and for sharing with us their work not intended for publication, which has independently developed estimators for this problem with identical rates of convergence.  
The estimator in \Cref{prop:stable-est-lr} and its proof were influenced and developed by multi-turn interactions with GPT-5.6 Sol Pro. The model also assisted with the computational implementation. The authors developed further results, provided simplification of arguments, independently checked the mathematical argument, 
and wrote the final manuscript. Thus, this work falls into a line of work 
where AI models have helped statisticians
make progress on unresolved problems, see e.g., \cite{dobriban2025solving,dobriban2026benjamini,AronowLopatto2026}, etc.


%


{\small
\setlength{\bibsep}{0.2pt plus 0.3ex}
\bibliographystyle{plainnat}
\bibliography{references}

@article{von1941mean,
  title={The mean square successive difference},
  author={Von Neumann, John and Kent, Robert H and Bellinson, HR and Hart, BI},
  journal={The Annals of Mathematical Statistics},
  volume={12},
  number={2},
  pages={153--162},
  year={1941},
  publisher={JSTOR}
}

@article{dobriban2026benjamini,
  title={The Benjamini--Hochberg Procedure Can Fail to Control the FDR for Correlated Two-Sided Gaussian Tests},
  author={Dobriban, Edgar},
  journal={arXiv preprint arXiv:2607.12208},
  year={2026}
}

@article{dobriban2025solving,
  title={Solving a Research Problem in Mathematical Statistics with AI Assistance},
  author={Dobriban, Edgar},
  journal={arXiv preprint arXiv:2511.18828},
  year={2025}
}

@article{stone1984asymptotically,
  title={An asymptotically optimal window selection rule for kernel density estimates},
  author={Stone, Charles J},
  journal={The Annals of Statistics},
  pages={1285--1297},
  year={1984},
  publisher={JSTOR}
}

@article{PaceBarry1997,
  author = {Pace, R. Kelley and Barry, Ronald},
  title = {Sparse Spatial Autoregressions},
  journal = {Statistics \& Probability Letters},
  volume = {33},
  number = {3},
  pages = {291--297},
  year = {1997},
  doi = {10.1016/S0167-7152(96)00140-X}
}

@misc{Rana2013CASP,
  author = {Rana, Prashant},
  title = {Physicochemical Properties of Protein Tertiary Structure},
  year = {2013},
  howpublished = {UCI Machine Learning Repository},
  doi = {10.24432/C5QW3H},
  note = {Dataset}
}

@article{newey2018cross,
  title={Cross-fitting and fast remainder rates for semiparametric estimation},
  author={Newey, Whitney K and Robins, James R},
  journal={arXiv preprint arXiv:1801.09138},
  year={2018}
}

@article{mcclean2026double,
  title={Double cross-fit doubly robust estimators: Beyond series regression},
  author={McClean, Alec and Balakrishnan, Sivaraman and Kennedy, Edward H and Wasserman, Larry},
  journal={Journal of the Royal Statistical Society Series B: Statistical Methodology},
  pages={qkag057},
  year={2026},
  publisher={Oxford University Press UK}
}

@incollection{RobinsEtAl2008,
  author = {Robins, James M. and Li, Lingling and Tchetgen Tchetgen, Eric J. and van der Vaart, Aad W.},
  title = {Higher Order Influence Functions and Minimax Estimation of Nonlinear Functionals},
  booktitle = {Probability and Statistics: Essays in Honor of David A. Freedman},
  series = {Institute of Mathematical Statistics Collections},
  volume = {2},
  pages = {335--421},
  year = {2008},
  publisher = {Institute of Mathematical Statistics},
  doi = {10.1214/193940307000000527}
}

@article{MunkEtAl2005,
  author = {Munk, Axel and Bissantz, Nicolai and Wagner, Thorsten and Freitag, Gudrun},
  title = {On Difference-Based Variance Estimation in Nonparametric Regression When the Covariate Is High Dimensional},
  journal = {Journal of the Royal Statistical Society: Series B},
  volume = {67},
  number = {1},
  pages = {19--41},
  year = {2005},
  doi = {10.1111/j.1467-9868.2005.00486.x}
}

@article{WangEtAl2008,
  author = {Wang, Lie and Brown, Lawrence D. and Cai, T. Tony and Levine, Michael},
  title = {Effect of Mean on Variance Function Estimation in Nonparametric Regression},
  journal = {The Annals of Statistics},
  volume = {36},
  number = {2},
  pages = {646--664},
  year = {2008},
  doi = {10.1214/009053607000000901}
}

@article{CaiLevineWang2009,
  author = {Cai, T. Tony and Levine, Michael and Wang, Lie},
  title = {Variance Function Estimation in Multivariate Nonparametric Regression},
  journal = {Journal of Multivariate Analysis},
  volume = {100},
  number = {1},
  pages = {126--136},
  year = {2009},
  doi = {10.1016/j.jmva.2008.03.007}
}

@article{RichardsonRotnitzky2014,
  author = {Richardson, Thomas S. and Rotnitzky, Andrea},
  title = {Causal Etiology of the Research of James M. Robins},
  journal = {Statistical Science},
  volume = {29},
  number = {4},
  pages = {459--484},
  year = {2014},
  doi = {10.1214/14-STS505}
}

@article{ShenEtAl2020,
  author = {Shen, Yandi and Gao, Chao and Witten, Daniela and Han, Fang},
  title = {Optimal Estimation of Variance in Nonparametric Regression with Random Design},
  journal = {The Annals of Statistics},
  volume = {48},
  number = {6},
  pages = {3589--3618},
  year = {2020},
  doi = {10.1214/20-AOS1944}
}

@misc{AronowLopatto2026,
  author = {Aronow, P. M. and Lopatto, Patrick},
  title = {On Rates Attainable under Random Design: A Negative Answer to a Problem of Robins},
  year = {2026},
  eprint = {2607.13170},
  archivePrefix = {arXiv},
  primaryClass = {math.ST}
}

@article{Rice1984,
  author = {Rice, John A.},
  title = {Bandwidth Choice for Nonparametric Regression},
  journal = {The Annals of Statistics},
  volume = {12},
  number = {4},
  pages = {1215--1230},
  year = {1984}
}

@article{HallKayTitterington1990,
  author = {Hall, Peter and Kay, J. W. and Titterington, D. M.},
  title = {Asymptotically Optimal Difference-Based Estimation of Variance in Nonparametric Regression},
  journal = {Biometrika},
  volume = {77},
  number = {3},
  pages = {521--528},
  year = {1990}
}

@article{BrownLevine2007,
  author = {Brown, Lawrence D. and Levine, Michael},
  title = {Variance Estimation in Nonparametric Regression via the Difference Sequence Method},
  journal = {The Annals of Statistics},
  volume = {35},
  number = {5},
  pages = {2219--2232},
  year = {2007},
  doi = {10.1214/009053607000000145}
}

@article{MullerSchickWefelmeyer2003,
  author = {M{\"u}ller, Ursula U. and Schick, Anton and Wefelmeyer, Wolfgang},
  title = {Estimating the Error Variance in Nonparametric Regression by a Covariate-Matched {U}-Statistic},
  journal = {Statistics},
  volume = {37},
  number = {3},
  pages = {179--188},
  year = {2003},
  doi = {10.1080/0233188031000078051}
}

@article{HallMarron1990,
  author = {Hall, Peter and Marron, J. S.},
  title = {On Variance Estimation in Nonparametric Regression},
  journal = {Biometrika},
  volume = {77},
  number = {2},
  pages = {415--419},
  year = {1990},
  doi = {10.1093/biomet/77.2.415}
}

@article{TongWang2005,
  author = {Tong, Tiejun and Wang, Yuedong},
  title = {Estimating Residual Variance in Nonparametric Regression Using Least Squares},
  journal = {Biometrika},
  volume = {92},
  number = {4},
  pages = {821--830},
  year = {2005},
  doi = {10.1093/biomet/92.4.821}
}

@article{TongMaWang2013,
  author = {Tong, Tiejun and Ma, Yanyuan and Wang, Yuedong},
  title = {Optimal Variance Estimation without Estimating the Mean Function},
  journal = {Bernoulli},
  volume = {19},
  number = {5A},
  pages = {1839--1854},
  year = {2013},
  doi = {10.3150/12-BEJ432}
}

@article{DuSchick2009,
  author = {Du, Jichang and Schick, Anton},
  title = {A Covariate-Matched Estimator of the Error Variance in Nonparametric Regression},
  journal = {Journal of Nonparametric Statistics},
  volume = {21},
  number = {3},
  pages = {263--285},
  year = {2009},
  doi = {10.1080/10485250802626873}
}

@article{LiitiainenEtAl2010,
  author = {Liiti{\"a}inen, Elia and Corona, Francesco and Lendasse, Amaury},
  title = {Residual Variance Estimation Using a Nearest Neighbor Statistic},
  journal = {Journal of Multivariate Analysis},
  volume = {101},
  number = {4},
  pages = {811--823},
  year = {2010},
  doi = {10.1016/j.jmva.2009.12.020}
}

@article{LiLin2020,
  author = {Li, Zhijian and Lin, Wei},
  title = {Efficient Error Variance Estimation in Non-Parametric Regression},
  journal = {Australian \& New Zealand Journal of Statistics},
  volume = {62},
  number = {4},
  pages = {467--484},
  year = {2020},
  doi = {10.1111/anzs.12311}
}
}

\appendix

\section{Proofs for Sections \ref{sec:model} and \ref{sec:est_low_reg}}

Throughout the appendix, constants may depend on the fixed model-class constants, $d$, $\beta_b$, $\beta_g$, $\lambda_0$, the fixed screening and ridge parameters, and the polynomial basis, but not on $n$ or on the particular parameter $\theta$.

\subsection{Existence and deterministic bias bound}

\begin{proof}[Proof of \Cref{lem:existence_aqstar}]
Let
\begin{equation*}
\mathcal A_Q
=
\left\{
\mathbf a\in\R^{N_Q}:
\|\mathbf a\|_2=1,
\ \Phi_Q^\top\mathbf a=0
\right\}.
\end{equation*}
Since $\operatorname{rank}(\Phi_Q)\leq q$ and $N_Q\geq q+1$, the null space of $\Phi_Q^\top$ is nontrivial. Hence $\mathcal A_Q$ is nonempty. It is also a closed subset of the unit sphere, and therefore compact.

For $f\in\cH^{\beta_b}(L;U)$, recall $f_Q=(f(X_{i_1}),\ldots,f(X_{i_{N_Q}}))^\top$. Since $\|f\|_\infty\leq L$, one has $\|f_Q\|_2\leq L\sqrt{N_Q}$. Thus, for $\mathbf a,\mathbf b$ on the unit sphere,
\begin{align*}
&\left|
\{\mathbf a^\top f_Q\}^2
-\{\mathbf b^\top f_Q\}^2
\right|\\
&\qquad\leq
|\{\mathbf a-\mathbf b\}^\top f_Q|
\left
(|\mathbf a^\top f_Q|+|\mathbf b^\top f_Q|
\right)
\leq2L^2N_Q\|\mathbf a-\mathbf b\|_2.
\end{align*}
Taking the supremum over $f$ shows that $\mathcal B_Q$ is Lipschitz, hence continuous, on the unit sphere. It therefore attains its minimum on the compact set $\mathcal A_Q$.

A measurable minimizer can be fixed as follows. For each fixed $m=N_Q\geq q+1$, the restrictions of the H\"older ball form a separable subset of $C([0,1]^d)$ in the uniform metric, so the supremum in \eqref{eq:ideal-cell-objective} can be taken over a fixed countable dense subset. Moreover, the uniform bounds on $f$ and its gradient show that this objective is jointly continuous in the $m$ covariates and the unit vector $\mathbf a$. Partition the configuration space into the finitely many Borel strata on which $\operatorname{rank}(\Phi_Q)$ is fixed. On each stratum the orthogonal projector onto $\mathcal N(\Phi_Q^\top)$ is continuous, so its feasible unit-sphere correspondence is continuous and compact-valued. The measurable maximum theorem gives a Borel minimizer on each stratum. Combining these selections, and then the possible occupancies, gives an $X$-measurable choice of $\mathbf a_Q^\star$.
\end{proof}

\begin{proof}[Proof of \Cref{lem:pe-bias bound}]
Fix $t\in(0,1]$, a cell $Q$, and a pair $e=(i,j)$ satisfying the conditions of the lemma. Let $x_Q$ be the cell center, and let $T_Q$ be the degree-$\ell$ Taylor polynomial of the extension of $f$ at $x_Q$, expressed in normalized coordinates $u$. Define
\begin{equation*}
r_Q(u)=f(b_Q+A_Qu)-T_Q(u),
\qquad u\in[0,1]^d.
\end{equation*}
The H\"older assumption on $f$ and the bounded aspect ratio of the cells  $Q$ imply that there exists a constant $C>0$ such that
\begin{equation}\label{eq:appendix-taylor-remainder}
\sup_{u\in[0,1]^d}|r_Q(u)|\leq CH_n^{\beta_b},
\qquad
|r_Q(u)-r_Q(v)|\leq CH_n^{\beta_b}\|u-v\|_2.
\end{equation}
Indeed, the first inequality is the usual Taylor remainder bound. For the second, we apply the Taylor remainder bound to the gradient. Specifically, in the original coordinates, the gradient of the remainder is $O(H_n^{\beta_b-1})$, and multiplication by $A_Q$ when passing to normalized coordinates contributes another factor $O(H_n)$. Indeed, this is where $\beta_b>1$ is crucially used.

Let $\mathbf T_Q=(T_Q(U_i))_{i\in I_Q}$ and $\mathbf r_Q=(r_Q(U_i))_{i\in I_Q}$. Since $T_Q$ has degree at most $\ell$, $\mathbf T_Q$ lies in the column space of $\Phi_Q$, so $R_Q\mathbf T_Q=0$. Hence
$d_e^\top R_Qf_Q=d_e^\top R_Q\mathbf r_Q$. Writing
$\Delta\psi_e=\psi(U_i)-\psi(U_j)$, we obtain
\begin{equation}\label{eq:appendix-bias bound-decomposition}
d_e^\top R_Qf_Q
=
r_Q(U_i)-r_Q(U_j)
-\Delta\psi_e^\top G_Q^{-1}\Phi_Q^\top\mathbf r_Q.
\end{equation}
The first term is bounded by $CH_n^{\beta_b}t$ by \eqref{eq:appendix-taylor-remainder} and the close-pair distance condition. Moreover, since the basis is Lipschitz,
$\|\Delta\psi_e\|_2\leq Ct$. Also, based on the choice of the cell, we have
$\|G_Q^{-1}\|_{\op}\leq(\kappa N_Q)^{-1}$, while boundedness of the basis on compact sets and \eqref{eq:appendix-taylor-remainder} give
$\|\Phi_Q^\top\mathbf r_Q\|_2\leq CN_QH_n^{\beta_b}$. Thus the second term in \eqref{eq:appendix-bias bound-decomposition} is also $O(H_n^{\beta_b}t)$. Dividing by $\sqrt{v_e}\geq\sqrt{v_0}$ proves \eqref{eq:pe-bias bound-bound}. 
\end{proof}

\subsection{Availability of certificate cells}

\begin{proof}[Proof of \Cref{lem:pe-availability}]
Set $k=q+2$. We first construct a fixed configuration in the normalized cell that certifies the existence of an eligible projected pair.

First we note that the collection $\{\psi(u):u\in(0,1)^d\}$ spans $\R^q$. Otherwise, there would be a nonzero $a\in\R^q$ such that the polynomial $u\mapsto a^\top\psi(u)$ vanished on the open set $(0,1)^d$, and hence vanished identically, contradicting linear independence of the basis. We may therefore choose distinct points $z_1,\ldots,z_q\in(0,1)^d$ such that
$V_0=
\begin{pmatrix}
\psi(z_1)^\top\\
\vdots\\
\psi(z_q)^\top
\end{pmatrix}
$
is nonsingular. By continuity, there exist pairwise disjoint closed boxes $B_1,\ldots,B_q\subset(0,1)^d$ and $\eta>0$ such that
\begin{equation}\label{eq:certificate-unisolvency-new}
\sigma_{\min}
\begin{pmatrix}
\psi(u_1)^\top\\
\vdots\\
\psi(u_q)^\top
\end{pmatrix}
\geq\eta
\end{equation}
whenever $u_r\in B_r$, $r=1,\ldots,q$.

Choose a closed cube $B_0\subset(0,1)^d$ disjoint from $B_1,\ldots,B_q$. There are an inner cube $B_0^-\subset B_0$ of positive volume and $h_0>0$ such that
\begin{equation}\label{eq:certificate-close-volume-new}
\int_{B_0}\int_{B_0}
\ind\{\|u-v\|_\infty\leq h/2\}\,du\,dv
\geq |B_0^-|h^d
\end{equation}
for every $0<h\leq h_0$. Indeed, choose $B_0^-$ at positive distance from the boundary of $B_0$ and take $h_0$ so that $u+[-h/2,h/2]^d\subset B_0$ for all $u\in B_0^-$.

We now Poissonize the design. Let $\Pp^\circ$ denote the law of a Poisson point process on $[0,1]^d$ with intensity $np(x)\,dx$, and let $K_Q$ be its number of points in $Q$. In normalized coordinates, the restriction to $Q$ has intensity
$
\lambda_Q(u)
=
\frac{n}{J_n}p(b_Q+A_Qu),
u\in[0,1]^d.
$
Since $J_n\asymp n$ and $\underline p\leq p\leq\overline p$, there are constants $0<\lambda_-\leq\lambda_+<\infty$ such that
\begin{equation}\label{eq:poisson-cell-intensity-bounds-new}
\lambda_-\leq\lambda_Q(u)\leq\lambda_+
\end{equation}
uniformly over $Q$, $u$, and all sufficiently large $n$.

For a deterministic $0<h\leq h_0$, call $Q$ certified when its Poisson restriction contains exactly one point $u_r$ in each $B_r$, $r=1,\ldots,q$, exactly two points $u,v$ in $B_0$ satisfying $\|u-v\|_\infty\leq h/2$, and no other points. Let $I_Q(h)$ be the indicator of this event and $p_Q(h)=\E^\circ I_Q(h)$. The restrictions to different cells are independent. Moreover, writing $\Lambda_Q=\int_{[0,1]^d}\lambda_Q(u)\,du$, the Poisson Janossy formula and \eqref{eq:certificate-close-volume-new}--\eqref{eq:poisson-cell-intensity-bounds-new} give
\begin{align}
p_Q(h)
&=
e^{-\Lambda_Q}
\left\{\prod_{r=1}^q\int_{B_r}\lambda_Q(u)\,du\right\}
\frac12
\int_{B_0}\int_{B_0}
\ind\{\|u-v\|_\infty\leq h/2\}
\lambda_Q(u)\lambda_Q(v)\,du\,dv
\geq c_{\mathrm{cert}}h^d.
\label{eq:poisson-certificate-probability-new}
\end{align}
Consequently, for
$S_h=\sum_{Q\in\cQ_n}I_Q(h)$,
$\mu_h=\E^\circ S_h$,
one has
\begin{equation}\label{eq:poisson-certificate-mean-new}
\mu_h\geq cJ_nh^d\geq c'nh^d.
\end{equation}
In particular, a multiplicative Chernoff bound under $\Pp^\circ$ gives
$\Pp^\circ\{S_h<cnh^d\}
\leq\exp(-c'nh^d)$.

Each certified cell passes the deterministic screens. On a certified Poisson configuration, use the same definitions of $N_Q$, $\Phi_Q$, $G_Q$, $R_Q$, and $v_e$ as for a fixed-size point configuration, with $N_Q=K_Q=q+2$. Let $L_\psi$ be a Lipschitz constant for $\psi$, and choose
\begin{equation}\label{eq:certificate-screen-constants-new}
\kappa=\frac{\eta^2}{2(q+2)},
\qquad
v_0=1.
\end{equation}
The $q$ points in $B_1,\ldots,B_q$ form a matrix $V_Q$ satisfying $\lambda_{\min}(V_Q^\top V_Q)\geq\eta^2$ by \eqref{eq:certificate-unisolvency-new}. Since $G_Q\succeq V_Q^\top V_Q$ and $N_Q=q+2$,
$\lambda_{\min}(G_Q/N_Q)
\geq\frac{\eta^2}{q+2}
\geq\kappa$.
For the two points $u,v\in B_0$, put $\Delta\psi=\psi(u)-\psi(v)$. If $e$ is this pair, then
\begin{align*}
v_e
&=2-\Delta\psi^\top G_Q^{-1}\Delta\psi
\geq
2-\frac{\|\Delta\psi\|_2^2}{\eta^2}
\geq
2-\frac{L_\psi^2d}{4\eta^2}h^2.
\end{align*}
After decreasing $h_0$ if necessary, $v_e\geq1=v_0$. The pair also has normalized distance at most $h/2$, so every certified cell belongs to $\mathcal C_n(h)$.

It remains to de-Poissonize. Let $K=\sum_QK_Q$ be the total point count. Directly dividing the preceding Chernoff bound by $\Pp^\circ(K=n)$ would lose a factor of order $\sqrt n$. The following argument cancels that factor. Fix $t>0$, set
\begin{equation*}
\zeta_Q=\E^\circ e^{-tI_Q(h)}
=1-(1-e^{-t})p_Q(h),
\qquad
L_t=\E^\circ e^{-tS_h}=\prod_{Q\in\cQ_n}\zeta_Q,
\end{equation*}
and define the tilted law
$\frac{d\Pp_t}{d\Pp^\circ}
=
\frac{e^{-tS_h}}{L_t}$.
Write $\E_t$ for expectation under $\Pp_t$. The cell restrictions remain independent under $\Pp_t$, and
\begin{equation}\label{eq:poisson-tilt-normalizer-new}
L_t
\leq
\exp\{-(1-e^{-t})\mu_h\}.
\end{equation}

For de-Poissonization, we only need the point-probability bound
\begin{equation}\label{eq:tilted-total-anticoncentration-new}
\Pp_t(K=n)\leq\frac{C}{\sqrt n}.
\end{equation}
Indeed, certification requires $K_Q=q+2\geq3$, so $I_Q(h)=0$ when $K_Q\in\{0,1\}$. By \eqref{eq:poisson-cell-intensity-bounds-new} and $\zeta_Q\leq1$, there are constants $a_0,a_1>0$ such that
\begin{equation*}
\Pp_t(K_Q=0)=\frac{e^{-\Lambda_Q}}{\zeta_Q}\geq a_0,
\qquad
\Pp_t(K_Q=1)=\frac{\Lambda_Qe^{-\Lambda_Q}}{\zeta_Q}\geq a_1
\end{equation*}
uniformly in $Q$. Choose a fixed $0<\rho<\min\{2a_0,2a_1,1\}$. The law of each $K_Q$ under $\Pp_t$ is a mixture of a $\operatorname{Bernoulli}(1/2)$ law with weight $\rho$ and a residual probability law with weight $1-\rho$: subtracting mass $\rho/2$ from each of the atoms at zero and one leaves a nonnegative measure of mass $1-\rho$.

We sample from these mixtures independently across cells. Let $M$ be the number of cells assigned to the fair Bernoulli component, so $M\sim\operatorname{Binomial}(J_n,\rho)$. Conditional on the component assignments and the counts drawn from the residual components, $K$ is a deterministic shift of a $\operatorname{Binomial}(M,1/2)$ variable. By Stirling's formula, the largest atom of this binomial law is at most $C/\sqrt{M+1}$. Splitting according to whether $M\geq\rho J_n/2$, and applying Chebyshev's inequality to $M$, gives
\begin{align*}
\Pp_t(K=n)
&\leq
\frac{C}{\sqrt{1+\rho J_n/2}}
+\Pp_t\{M<\rho J_n/2\}
\leq
\frac{C}{\sqrt{1+\rho J_n/2}}
+\frac{4(1-\rho)}{\rho J_n}
\leq\frac{C}{\sqrt n},
\end{align*}
where we used $J_n\asymp n$. This proves \eqref{eq:tilted-total-anticoncentration-new} with constants uniform over the model class and $0<h\leq h_0$.

Take $s_h=c_0nh^d$, where $c_0>0$ is sufficiently small. By exponential Markov, \eqref{eq:poisson-certificate-mean-new}, \eqref{eq:poisson-tilt-normalizer-new}, and \eqref{eq:tilted-total-anticoncentration-new},
\begin{align}
\Pp^\circ\{S_h<s_h,\ K=n\}
&\leq
e^{ts_h}\E^\circ\{e^{-tS_h}\ind\{K=n\}\}
=
e^{ts_h}L_t\Pp_t(K=n)
\leq
\frac{C}{\sqrt n}\exp(-cnh^d).
\label{eq:joint-poisson-certificate-tail-new}
\end{align}
Here $t$ is any fixed positive constant and $c_0$ is chosen so that the positive term $ts_h$ is absorbed by $(1-e^{-t})\mu_h$.

Under $\Pp^\circ$, $K\sim\operatorname{Poisson}(n)$, and Stirling's formula gives $\Pp^\circ(K=n)\geq c/\sqrt n$. Conditional on $K=n$, the points are exactly an iid sample from $p$. Since $S_h\leq|\mathcal C_n(h)|$, \eqref{eq:joint-poisson-certificate-tail-new} therefore gives
\begin{align*}
\Pp\{|\mathcal C_n(h)|<c_0nh^d\}
&\leq
\Pp^\circ\{S_h<s_h\mid K=n\}
\leq
C\exp(-cnh^d).
\end{align*}
Applying this with $h=t_n$ proves \eqref{eq:pe-availability}.
\end{proof}

\subsection{A simpler availability bound}\label{sec:availability-chebyshev}

The exponential probability bound in \Cref{lem:pe-availability} is stronger than is needed for the risk bounds.  The following fixed-sample version is sufficient and avoids Poissonization entirely.

\begin{lemma}[Fixed-sample availability by a second-moment argument]\label{lem:pe-availability-chebyshev}
Choose $\kappa$ and $v_0$ as in \eqref{eq:certificate-screen-constants-new}. Whenever $t_n\to0$ and $nt_n^d\to\infty$, there are constants $c,C>0$ such that
\begin{equation}\label{eq:pe-availability-chebyshev}
\Pp\left\{
|\mathcal C_n(t_n)|<cnt_n^d
\right\}
\leq \frac{C}{nt_n^d}
\end{equation}
for all sufficiently large $n$, uniformly over $\theta\in\Theta(\beta_b,\beta_g,d)$.
\end{lemma}

\begin{proof}
Set $k=q+2$ and use the boxes $B_0,B_1,\ldots,B_q$ constructed in the proof of \Cref{lem:pe-availability}. For $Q=b_Q+A_Q[0,1]^d$ and a Borel set $A\subseteq[0,1]^d$, define
\begin{equation*}
\nu_Q(A)
=
\Pp\{X_1\in Q,\ A_Q^{-1}(X_1-b_Q)\in A\}
=
\frac1{J_n}\int_A p(b_Q+A_Qu)\,du,
\end{equation*}
and let $\pi_Q=\nu_Q([0,1]^d)=\Pp(X_1\in Q)$. The density bounds imply
\begin{equation}\label{eq:fixed-cell-measure-bounds}
\frac{\underline p}{J_n}|A|
\leq \nu_Q(A)\leq
\frac{\overline p}{J_n}|A|,
\qquad
\frac{\underline p}{J_n}
\leq \pi_Q\leq
\frac{\overline p}{J_n}.
\end{equation}

For a deterministic $0<h\leq h_0$, let $\xi_Q(h)$ be the indicator that the original fixed-size sample contains exactly one point in each normalized box $B_r$, $r=1,\ldots,q$, exactly two points $u,v$ in $B_0$ satisfying $\|u-v\|_\infty\leq h/2$, and no other point in $Q$. Put
\begin{align}
w_Q(h)
&=
\left\{\prod_{r=1}^q\nu_Q(B_r)\right\}
\frac12\int_{B_0}\int_{B_0}
\ind\{\|u-v\|_\infty\leq h/2\}
\,\nu_Q(du)\nu_Q(dv).
\label{eq:fixed-certificate-weight}
\end{align}
By \eqref{eq:certificate-close-volume-new} and \eqref{eq:fixed-cell-measure-bounds}, and by the corresponding immediate upper volume bound,
\begin{equation}\label{eq:fixed-certificate-weight-bounds}
cJ_n^{-k}h^d
\leq w_Q(h)\leq
CJ_n^{-k}h^d
\end{equation}
uniformly over $Q$ and the model class.

Letting $(n)_r=n(n-1)\cdots(n-r+1)$, a direct multinomial calculation gives the certificate probability
\begin{equation}\label{eq:fixed-certificate-probability}
\alpha_Q(h)
:=
\E\xi_Q(h)
=
(n)_k w_Q(h)(1-\pi_Q)^{n-k}.
\end{equation}
Because $J_n\asymp n$, \eqref{eq:fixed-cell-measure-bounds}--\eqref{eq:fixed-certificate-probability} yield
\begin{equation}\label{eq:fixed-certificate-probability-bounds}
ch^d\leq \alpha_Q(h)\leq Ch^d.
\end{equation}
Indeed, this follows from the facts that $(n)_kJ_n^{-k}$ is bounded above and away from zero, while
$(1-\pi_Q)^{n-k}$ is bounded away from zero uniformly for all sufficiently large $n$.

Although the indicators $\xi_Q(h)$ are not independent, 
we leverage that their covariances are small enough. For distinct cells $Q$ and $R$, another multinomial calculation gives
\begin{equation*}
\E\{\xi_Q(h)\xi_R(h)\}
=
(n)_{2k}w_Q(h)w_R(h)
(1-\pi_Q-\pi_R)^{n-2k}.
\end{equation*}
Consequently,
\begin{align}
\frac{\E\{\xi_Q(h)\xi_R(h)\}}
{\alpha_Q(h)\alpha_R(h)}
&=
\frac{(n)_{2k}}{(n)_k^2}
\frac{(1-\pi_Q-\pi_R)^{n-2k}}
{(1-\pi_Q)^{n-k}(1-\pi_R)^{n-k}}
\leq
\{(1-\pi_Q)(1-\pi_R)\}^{-k}
\leq 1+\frac{C}{n}.
\label{eq:fixed-certificate-pair-ratio}
\end{align}
Here we used
\begin{equation*}
\frac{(n)_{2k}}{(n)_k^2}
=\frac{(n-k)_k}{(n)_k}\leq1,
\qquad
1-\pi_Q-\pi_R
\leq(1-\pi_Q)(1-\pi_R),
\end{equation*}
followed by $\pi_Q\vee\pi_R\leq C/n$. It follows from \eqref{eq:fixed-certificate-pair-ratio} that
\begin{equation}\label{eq:fixed-certificate-covariance}
\operatorname{Cov}\{\xi_Q(h),\xi_R(h)\}
\leq
\frac{C}{n}\alpha_Q(h)\alpha_R(h).
\end{equation}
Note that we only need this upper bound since a negative covariance can only reduce the variance of the sum of indicators considered below. Specifically, let
$S_n(h)=\sum_{Q\in\cQ_n}\xi_Q(h)$,
$\mu_n(h)=\E S_n(h)$.
By \eqref{eq:fixed-certificate-probability-bounds} and $J_n\asymp n$,
\begin{equation}\label{eq:fixed-certificate-mean}
cnh^d\leq\mu_n(h)\leq Cnh^d.
\end{equation}
Moreover, for $h\leq 1$ \eqref{eq:fixed-certificate-covariance} gives
\begin{align}
\Var\{S_n(h)\}
&\leq
\sum_{Q\in\cQ_n}\alpha_Q(h)
+\frac{C}{n}
\sum_{Q\neq R}\alpha_Q(h)\alpha_R(h)
\leq
Cnh^d+Cn h^{2d}
\leq Cnh^d.
\label{eq:fixed-certificate-variance}
\end{align}
Choose $c_0>0$ small enough that $c_0nh^d\leq\mu_n(h)/2$. Chebyshev's inequality, \eqref{eq:fixed-certificate-mean}, and \eqref{eq:fixed-certificate-variance} imply
\begin{align}
\Pp\{S_n(h)<c_0nh^d\}
&\leq
\Pp\{|S_n(h)-\mu_n(h)|\geq\mu_n(h)/2\}
\leq
\frac{4\Var\{S_n(h)\}}{\mu_n(h)^2}
\leq \frac{C}{nh^d}.
\label{eq:fixed-certificate-chebyshev}
\end{align}
The deterministic verification in the proof of \Cref{lem:pe-availability} shows that every cell counted by $S_n(h)$ belongs to $\mathcal C_n(h)$. Hence $S_n(h)\leq|\mathcal C_n(h)|$. Taking $h=t_n$ in \eqref{eq:fixed-certificate-chebyshev} proves \eqref{eq:pe-availability-chebyshev}.
\end{proof}

To see why this weaker bound suffices for \Cref{prop:ideal_est_lr2}, decrease the  $c_0$ if necessary and take $t_n=\varepsilon_n$. 
Since $\mathcal T_n(\varepsilon_n)=\mathcal C_n(\varepsilon_n)$, the failure probability is at most $C/(n\varepsilon_n^d)$. The fallback has bounded squared error, so its contribution to the risk is of the same order as the conditional variance term in \Cref{lem:cell-average-risk}. Thus the risk remains bounded by $C\{H_n^{4\beta_b}\varepsilon_n^4+(n\varepsilon_n^d)^{-1}\}$, which gives the stated rate from \eqref{eq:lr-optimal-bandwidth}.

\subsection{The ideal and feasible witness estimators}

\begin{proof}[Proof of \Cref{lem:enough_good_Q}]
Use the constants from \Cref{lem:pe-availability} and apply that lemma with $t_n=\varepsilon_n$. For every $Q\in\mathcal C_n(\varepsilon_n)$, choose any eligible pair $e$. The vector
$\mathbf a_{Q,e}
=
\frac{R_Qd_e}{\sqrt{d_e^\top R_Qd_e}}
$ has unit norm and satisfies $\Phi_Q^\top\mathbf a_{Q,e}=0$. By \Cref{lem:pe-bias bound}, uniformly over $f\in\cH^{\beta_b}(L;U)$,
$(\mathbf a_{Q,e}^\top f_Q)^2
\leq CH_n^{2\beta_b}\varepsilon_n^2$.
Since $\mathbf a_Q^\star$ minimizes the worst-case conditional bias over all feasible unit contrasts,
\begin{equation*}
\mathcal B_Q(\mathbf a_Q^\star)
\leq
\mathcal B_Q(\mathbf a_{Q,e})
\leq CH_n^{2\beta_b}\varepsilon_n^2.
\end{equation*}
Thus $\mathcal C_n(\varepsilon_n)\subseteq\mathcal S_n(\varepsilon_n,C)$ after increasing $C$ if necessary. The conclusion follows directly from \eqref{eq:pe-availability}.
\end{proof}

We use the following elementary conditional risk bound twice.

\begin{lemma}\label{lem:cell-average-risk}
Let $\mathcal D$ be an $X$-measurable collection of disjoint cells, each containing $q+2$ observations. For every $Q\in\mathcal D$, let $\mathbf a_Q$ be an $X$-measurable unit vector supported on that cell. Suppose
$(\mathbf a_Q^\top f_Q)^2\leq B_n$
for every $Q\in\mathcal D$. On the event $|\mathcal D|\geq m_n\geq1$,
\begin{equation}\label{eq:cell-average-risk-bound}
\E\left[
\left\{
\frac1{|\mathcal D|}
\sum_{Q\in\mathcal D}(\mathbf a_Q^\top Y_Q)^2
-\sigma^2
\right\}^2
\mathrel{\Big|}X
\right]
\leq B_n^2+\frac{C}{m_n}.
\end{equation}
\end{lemma}

\begin{proof}
Write $\mu_Q=\mathbf a_Q^\top f_Q$ and $\eta_Q=\mathbf a_Q^\top\varepsilon_Q$. Conditional on $X$,
$\E(\eta_Q)=0$ and $\E(\eta_Q^2)=\sigma^2$ because $\|\mathbf a_Q\|_2=1$. Hence
$\E\{(\mathbf a_Q^\top Y_Q)^2\mid X\}
=\sigma^2+\mu_Q^2$.
The conditional bias of the average is therefore at most $B_n$.

For the variance, conditional independence and centering of the errors give
\begin{align*}
\E(\eta_Q^4\mid X)
&=
\sum_r a_{Q,r}^4\E(\varepsilon_{i_r}^4\mid X)
+6\sum_{r<s}a_{Q,r}^2a_{Q,s}^2\sigma^4
\leq M_4+3\overline v^{\,2}.
\end{align*}
Also $|\mu_Q|\leq L\sqrt{q+2}$. Thus
$\E\{(\mathbf a_Q^\top Y_Q)^4\mid X\}\leq C$, uniformly in $Q$ and $\theta$. The cellwise data are conditionally independent across disjoint cells, so the conditional variance of their average is at most $C/|\mathcal D|\leq C/m_n$. Combining conditional squared bias and variance proves \eqref{eq:cell-average-risk-bound}.
\end{proof}

\begin{proof}[Proof of \Cref{prop:ideal_est_lr}]
We use the fixed-sample bound in \Cref{lem:pe-availability-chebyshev}. Decrease $c_0$ if necessary so that this lemma applies. The deterministic argument in the proof of \Cref{lem:enough_good_Q} gives
$
\mathcal C_n(\varepsilon_n)\subseteq\mathcal S_n(\varepsilon_n,C)$
after increasing $C$ if necessary. Therefore,
\begin{equation}\label{eq:ideal-good-event-chebyshev}
\Pp\{\mathcal G_n(\varepsilon_n)^c\}
\leq
\Pp\left\{
|\mathcal C_n(\varepsilon_n)|<c_0n\varepsilon_n^d
\right\}
\leq
\frac{C}{n\varepsilon_n^d}.
\end{equation}

On $\mathcal G_n(\varepsilon_n)$, apply \Cref{lem:cell-average-risk} with
$\mathcal D=\mathcal S_n(\varepsilon_n,C),
B_n=CH_n^{2\beta_b}\varepsilon_n^2,
m_n=c_0n\varepsilon_n^d$.
This gives
\begin{equation*}
\E\{(\widehat\sigma_{n,\mathrm{LR}}^2-\sigma^2)^2\mid X\}
\leq
CH_n^{4\beta_b}\varepsilon_n^4
+\frac{C}{n\varepsilon_n^d}
\end{equation*}
on that event. On its complement, the fallback lies in $[\underline v,\overline v]$, so its squared error is bounded by a constant. Integrating over the design and using \eqref{eq:ideal-good-event-chebyshev},
\begin{equation}\label{eq:ideal-risk-before-balance}
\sup_{\theta\in\Theta(\beta_b,\beta_g,d)}
\E_\theta(\widehat\sigma_{n,\mathrm{LR}}^2-\sigma^2)^2
\leq
C\left\{
H_n^{4\beta_b}\varepsilon_n^4
+\frac{1}{n\varepsilon_n^d}
\right\}.
\end{equation}
Since $H_n\asymp n^{-1/d}$ and
$\varepsilon_n\asymp n^{-(d-4\beta_b)/\{d(d+4)\}}$,
$H_n^{4\beta_b}\varepsilon_n^4
\asymp
\frac1{n\varepsilon_n^d}
\asymp
n^{-4(\beta_b+1)/(d+4)}$.
The contribution of $\mathcal G_n(\varepsilon_n)^c$ is absorbed into the term $(n\varepsilon_n^d)^{-1}$ by \eqref{eq:ideal-good-event-chebyshev}. This proves \eqref{eq:ideal-estimator-rate}.
\end{proof}

\begin{proof}[Proof of \Cref{prop:ideal_est_lr2}]
Choose $\kappa$ and $v_0$ as in \eqref{eq:certificate-screen-constants-new}, and decrease $c_0$ if necessary so that it is valid in \Cref{lem:enough_good_Q,lem:pe-availability,lem:pe-availability-chebyshev}. Then
$\mathcal T_n(\varepsilon_n)=\mathcal C_n(\varepsilon_n)$ because the two collections use the same screens. Thus
\begin{equation}\label{eq:feasible-good-event-chebyshev}
\Pp\{\widetilde{\mathcal G}_n(\varepsilon_n)^c\}
\leq \frac{C}{n\varepsilon_n^d}
\end{equation}
by \Cref{lem:pe-availability-chebyshev}.
For every retained cell, \Cref{lem:pe-bias bound} gives
$(\widetilde{\mathbf a}_Q^\top f_Q)^2
\leq CH_n^{2\beta_b}\varepsilon_n^2$.
Apply \Cref{lem:cell-average-risk} on $\widetilde{\mathcal G}_n(\varepsilon_n)$ and use the bounded fallback on the complement. Integrating over the design and using \eqref{eq:feasible-good-event-chebyshev} yields exactly the right-hand side of \eqref{eq:ideal-risk-before-balance}. The balancing calculation in the proof of \Cref{prop:ideal_est_lr} therefore proves \eqref{eq:feasible-estimator-rate}.
\end{proof}

Alternatively, one can use the exponential availability bounds in \Cref{lem:enough_good_Q,lem:pe-availability} in the preceding two proofs. They give
$\Pp\{\mathcal G_n(\varepsilon_n)^c\}
+\Pp\{\widetilde{\mathcal G}_n(\varepsilon_n)^c\}
\leq C\exp(-cn\varepsilon_n^d)$.
The same conditional risk calculation then bounds the mean squared error of each estimator by
\begin{equation}\label{eq:lr-risk-exponential-alternative}
C\left\{
H_n^{4\beta_b}\varepsilon_n^4
+\frac{1}{n\varepsilon_n^d}
+\exp(-cn\varepsilon_n^d)
\right\}.
\end{equation}
At \eqref{eq:lr-optimal-bandwidth}, $n\varepsilon_n^d\asymp n^{4(\beta_b+1)/(d+4)}\to\infty$, so the exponential term is negligible and this argument gives the same rates.

\subsection{Proof of the main theorem}

\begin{proof}[Proof of \Cref{thm:main}]
\Cref{prop:ideal_est_lr2} proves \eqref{eq:feasible-estimator-rate} under conditions (i)--(iii) for every $\beta_g>0$, with $\varepsilon_n$ chosen according to \eqref{eq:lr-optimal-bandwidth}. Its proof uses only the upper and lower bounds on the design density; the H\"older smoothness of $p$ is not used. In particular, when $0<\beta_g\leq\beta_g^\star$, the definition $\widehat\sigma_{n,\mathrm{sharp}}^2=\widetilde\sigma_{n,\mathrm{LR}}^2$ gives the second line of \eqref{eq:sharp-smooth-piecewise}.

Now suppose that $\beta_g>\beta_g^\star$. Under the stated specialization of \citet{RobinsEtAl2008}, their condition~(4.11) reduces to
\begin{equation*}
\beta_g>
\frac{\beta_b(1-4\beta_b/d)}
{1+2\beta_b/d+8(\beta_b/d)^2},
\end{equation*}
which is precisely $\beta_g>\beta_g^\star$. Let $\mathcal T$ denote the training subsample used for the regression and density pilots in the cited construction; the training and estimation sample sizes are fixed positive proportions of the total sample size $n$. In addition to the applicability conditions of that construction, the high-regularity mean squared-error comparison assumes the following uniform integrated bounds:
\begin{equation}\label{eq:hoif-uniform-moment-condition}
\begin{aligned}
\sup_{\theta\in\Theta(\beta_b,\beta_g,d)}
\E_\theta\left[
\left\{\E_\theta(\widetilde\sigma_{n,\mathrm{HOIF}}^2\mid\mathcal T)-\sigma^2\right\}^2
\right]
&\leq Cn^{-8\beta_b/(d+4\beta_b)},\\
\sup_{\theta\in\Theta(\beta_b,\beta_g,d)}
\E_\theta\left[
\Var_\theta(\widetilde\sigma_{n,\mathrm{HOIF}}^2\mid\mathcal T)
\right]
&\leq Cn^{-8\beta_b/(d+4\beta_b)}.
\end{aligned}
\end{equation}

The conditional stochastic-order bounds in the cited result alone do not imply \eqref{eq:hoif-uniform-moment-condition}; the latter is an additional assumption of the comparison. Given this assumption, the conditional bias--variance decomposition gives
\begin{align*}
&\sup_{\theta\in\Theta(\beta_b,\beta_g,d)}
\E_\theta
(\widetilde\sigma_{n,\mathrm{HOIF}}^2-\sigma^2)^2
=
\sup_{\theta\in\Theta(\beta_b,\beta_g,d)}
\E_\theta\left[
\Var_\theta(\widetilde\sigma_{n,\mathrm{HOIF}}^2\mid\mathcal T)
+
\left\{\E_\theta(\widetilde\sigma_{n,\mathrm{HOIF}}^2\mid\mathcal T)-\sigma^2\right\}^2
\right]\\
&\qquad\leq Cn^{-8\beta_b/(d+4\beta_b)}.
\end{align*}
Clipping cannot increase squared error, so the same bound holds for $\widehat\sigma_{n,\mathrm{HOIF}}^2=\widehat\sigma_{n,\mathrm{sharp}}^2$. This proves the first line of \eqref{eq:sharp-smooth-piecewise}.
\end{proof}

\section{A Numerically Stable All-Pairs Extension}\label{sec:est_num_stable}

The feasible witness estimator in the main text keeps one witness pair from each cell containing exactly $q+2$ observations. This section develops an optional numerically stable extension, evaluated alongside the feasible witness estimator in the experiments, that permits a bounded range of occupancies and reuses all eligible pairs.

For a cell $Q$, let $I_Q=\{i:X_i\in Q\}$, $N_Q=|I_Q|$, $\Phi_Q=(\psi(U_i)^\top)_{i\in I_Q}$, and $G_Q=\Phi_Q^\top\Phi_Q$. Fix an integer $N_{\max}\geq q+2$ and constants $\kappa>0$, $v_0>0$, and $\lambda>0$. Retain $Q$ when
$q<N_Q\leq N_{\max}$,
$\lambda_{\min}(N_Q^{-1}G_Q)\geq\kappa$,
and, in a retained cell, define the residual projections
$   P_Q=\Phi_QG_Q^{-1}\Phi_Q^\top,
 R_Q=I_{N_Q}-P_Q$.

Let $h_n\in(0,1]$ be a normalized scale parameter to be specified later. For an unordered pair $e=(i,j)$ in a retained cell, let $e_i,e_j\in\R^{N_Q}$ denote the coordinate vectors corresponding to $i,j$ in the local ordering of $I_Q$, set $d_e=e_i-e_j$, and retain the pair when its two components are sufficiently close and the projected contrast is nondegenerate:
\begin{equation}\label{eq:stable-pair-screen}
    \|U_i-U_j\|_\infty\leq h_n,
    \qquad
    v_e=d_e^\top R_Qd_e\geq v_0.
\end{equation}

We now use the pair-anchored projected contrast $R_Qd_e$ as the key object. Define the quadratic statistic and normalized squared feature distance
\begin{equation}\label{eq:stable-pair-statistic}
    Z_e=\frac{(d_e^\top R_QY_Q)^2}{v_e},
    \qquad
    s_e=\frac{\|U_i-U_j\|_2^2}{h_n^2}.
\end{equation}
Since $R_Q$ is an orthogonal projection,
\begin{equation}\label{eq:stable-pair-conditional-mean}
    \E(Z_e\mid X)
    =\sigma^2+\frac{(d_e^\top R_Qf_Q)^2}{v_e}.
\end{equation}
This shows why the construction improves on the close-pair estimator: the expression $d_e^\top R_Qf_Q$ removes the local polynomial trend of degree $\ell$, thereby reducing the bias.

Let $\mathcal E_n(Q)$ be the set of retained pairs in $Q$, let $\mathcal E_n=\bigcup_{Q\in\cQ_n}\mathcal E_n(Q)$, and let $M_n=|\mathcal E_n|$. When $M_n\geq1$, set $\bar s=M_n^{-1}\sum_{e\in\mathcal E_n}s_e$ and $S_s=\sum_{e\in\mathcal E_n}(s_e-\bar s)^2$, and use the ridge-intercept weights
\begin{equation}\label{eq:pe-weights}
    w_e=\frac1{M_n}-\frac{\bar s(s_e-\bar s)}{S_s+\lambda M_n}.
\end{equation}
The estimator is
\begin{equation}\label{eq:pe-estimator}
    \widetilde\sigma_{n,\mathrm{PE}}^2
    =
    \begin{cases}
    \displaystyle\sum_{e\in\mathcal E_n}w_eZ_e,&M_n\geq1,\\
    \underline v,&M_n=0,
    \end{cases}
    \qquad
    \widehat\sigma_{n,\mathrm{PE}}^2
    =\clip_{[\underline v,\overline v]}(\widetilde\sigma_{n,\mathrm{PE}}^2).
\end{equation}
Equivalently, pool all retained pairs, fit
\begin{equation*}
    (\widehat a,\widehat b)
    =\operatorname*{argmin}_{a,b}
    \left[
    \sum_{e\in\mathcal E_n}\{Z_e-a-b(s_e-\bar s)\}^2
    +\lambda M_nb^2
    \right],
\end{equation*}
and return the estimated intercept $\widehat a-\bar s\widehat b$. Thus the displayed estimator performs one global ridge regression over all retained pairs, rather than separate within-cell regressions followed by aggregation. For fixed $\lambda>0$, the weights do not exactly cancel a linear distance term: $\sum_e w_es_e=\bar s\lambda M_n/(S_s+\lambda M_n)$. The risk analysis uses weight stability, not exact extrapolation-bias cancellation.
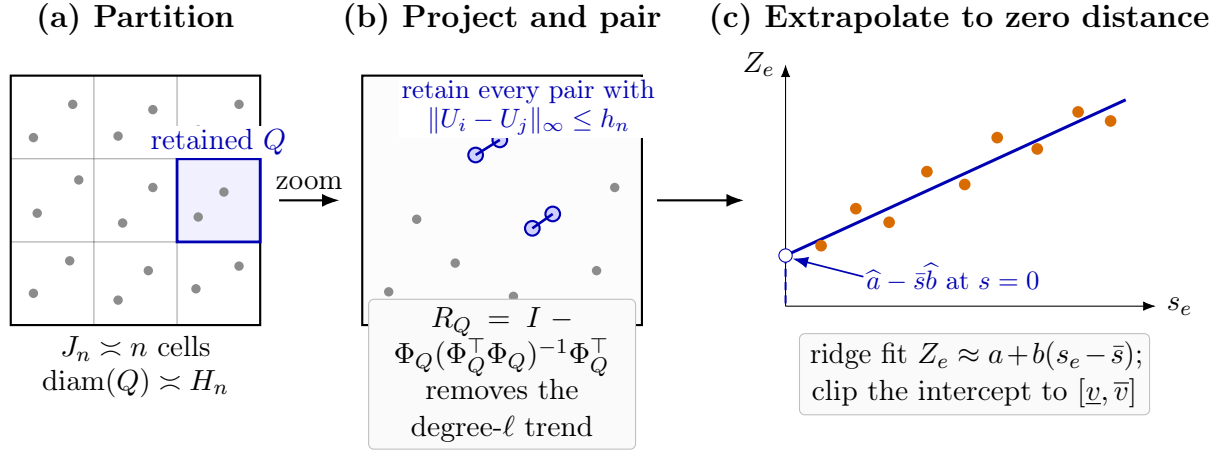
\begin{figure}
\centering
\begin{tikzpicture}[
  x=1cm,y=1cm,font=\small,>=Latex,
  sample/.style={circle,fill=black!45,inner sep=1.25pt},
  pair/.style={circle,draw=blue!70!black,fill=blue!18,
               line width=0.8pt,minimum size=5.5pt,inner sep=0pt},
  box/.style={draw=black!25,rounded corners=2pt,fill=black!2,
              inner xsep=4pt,inner ysep=3pt,align=center}
]
\begin{scope}
  \node[font=\bfseries] at (1.65,4.05) {(a) Partition};
  \fill[blue!6] (2.2,1.1) rectangle (3.3,2.2);
  \draw[black!38,line width=0.35pt] (0,0) grid[step=1.1] (3.3,3.3);
  \draw[black,line width=0.8pt] (0,0) rectangle (3.3,3.3);
  \draw[blue!70!black,line width=1.05pt] (2.2,1.1) rectangle (3.3,2.2);
  \foreach \p in {(0.30,0.42),(0.78,0.85),(1.42,0.33),(1.82,0.72),
                   (2.45,0.48),(3.02,0.78),(0.35,1.48),(0.90,1.92),
                   (1.48,1.35),(1.88,1.82),(2.48,1.43),(2.82,1.76),
                   (0.30,2.48),(0.82,2.92),(1.42,2.50),(1.88,2.85),
                   (2.48,2.58),(3.02,2.92)}
    \node[sample] at \p {};
  \node[blue!70!black,fill=white,inner sep=1pt] at (2.75,2.42)
    {retained $Q$};
  \node[align=center] at (1.65,-0.55)
    {$J_n\asymp n$ cells\\[-1pt]$\operatorname{diam}(Q)\asymp H_n$};
\end{scope}

\draw[->,line width=0.8pt] (3.55,1.65)--(4.35,1.65)
  node[midway,above] {zoom};

\begin{scope}[xshift=4.65cm]
  \node[font=\bfseries] at (1.85,4.05) {(b) Project and pair};
  \fill[black!1] (0,0) rectangle (3.7,3.3);
  \draw[black,line width=0.8pt] (0,0) rectangle (3.7,3.3);
  \foreach \p in {(0.35,0.44),(0.72,1.40),(0.68,2.67),(1.23,0.82),
                   (1.50,2.25),(2.00,0.38),(2.25,1.28),(2.55,2.72),
                   (3.10,0.72),(3.34,1.82)}
    \node[sample] at \p {};
  \coordinate (p1) at (1.50,2.25);
  \coordinate (p2) at (1.82,2.45);
  \coordinate (p3) at (2.25,1.28);
  \coordinate (p4) at (2.52,1.47);
  \node[pair] at (p1) {};
  \node[pair] at (p2) {};
  \node[pair] at (p3) {};
  \node[pair] at (p4) {};
  \draw[blue!70!black,line width=1pt] (p1)--(p2);
  \draw[blue!70!black,line width=1pt] (p3)--(p4);
  \node[box,text width=3.25cm] at (1.85,-0.65)
    {$R_Q=I-\Phi_Q(\Phi_Q^\top\Phi_Q)^{-1}\Phi_Q^\top$\\
     removes the degree-$\ell$ trend};
  \node[align=center,blue!70!black,fill=white,inner sep=1pt,
        font=\footnotesize] at (2.20,2.88)
    {retain every pair with\\[-1pt]$\|U_i-U_j\|_\infty\leq h_n$};
\end{scope}

\draw[->,line width=0.8pt] (8.55,1.65)--(9.70,1.65);

\begin{scope}[xshift=10cm]
  \node[font=\bfseries] at (2.55,4.05) {(c) Extrapolate to zero distance};
  \draw[->] (0.25,0.25)--(5.15,0.25) node[right] {$s_e$};
  \draw[->] (0.25,0.25)--(0.25,3.45) node[left] {$Z_e$};
  \draw[blue!70!black,line width=1.15pt] (0.25,0.92)--(4.75,2.98);
  \draw[blue!70!black,densely dashed,line width=0.75pt] (0.25,0.92)--(0.25,0.25);
  \foreach \p in {(0.72,1.05),(1.18,1.54),(1.62,1.36),(2.12,2.03),
                   (2.62,1.86),(3.05,2.48),(3.58,2.33),(4.12,2.82),
                   (4.55,2.70)}
    \node[circle,fill=orange!85!black,inner sep=1.55pt] at \p {};
  \node[circle,draw=blue!70!black,fill=white,minimum size=5pt,inner sep=0pt]
    at (0.25,0.92) {};
  \draw[->,blue!70!black,line width=0.65pt] (1.28,0.64)--(0.34,0.88);
  \node[anchor=west,blue!70!black,fill=white,inner sep=1pt,
        font=\footnotesize] at (1.28,0.64)
    {$\widehat a-\bar s\widehat b$ at $s=0$};
  \node[box,text width=4.35cm] at (2.75,-0.65)
    {ridge fit $Z_e\approx a+b(s_e-\bar s)$;\\
     clip the intercept to $[\underline v,\overline v]$};
\end{scope}
\end{tikzpicture}
\caption{Geometry of the numerically stable extension, shown schematically in two dimensions. The sample is partitioned into $J_n\asymp n$ cells. In each cell passing the occupancy and Gram screens, $R_Q$ removes the local degree-$\ell$ polynomial component. Every pair passing the distance and residual-variance screens contributes $Z_e=(d_e^\top R_QY_Q)^2/v_e$ at $s_e=\|U_i-U_j\|_2^2/h_n^2$. One global ridge line is extrapolated to $s=0$, and its intercept is clipped to $[\underline v,\overline v]$.}
\label{fig:projection-extrapolation}
\end{figure}

\begin{proposition}[Numerically stable extension]\label{prop:stable-est-lr}
Under conditions (i)--(iii), with $\beta_b>1$ and $d>4\beta_b$, there are fixed choices of $N_{\max}$, $\kappa$, and $v_0$ such that, for every fixed $\lambda>0$ and every $h_n\to0$ with $nh_n^d\to\infty$,
\begin{equation}\label{eq:stable-generic-risk}
\sup_{\theta\in\Theta(\beta_b,\beta_g,d)}
\E_\theta
(\widehat\sigma_{n,\mathrm{PE}}^2-\sigma^2)^2
\leq
C\left\{
\frac1{nh_n^d}
+H_n^{4\beta_b}h_n^4
+\exp(-cnh_n^d)
\right\}.
\end{equation}
Consequently, if
\begin{equation}\label{eq:stable-optimal-bandwidth}
h_n\asymp n^{-(d-4\beta_b)/\{d(d+4)\}},
\end{equation}
then
\begin{equation}\label{eq:stable-optimal-rate}
\sup_{\theta\in\Theta(\beta_b,\beta_g,d)}
\E_\theta
(\widehat\sigma_{n,\mathrm{PE}}^2-\sigma^2)^2
\leq Cn^{-4(\beta_b+1)/(d+4)}.
\end{equation}
The same bounds hold for the unclipped estimator $\widetilde\sigma_{n,\mathrm{PE}}^2$.
\end{proposition}

\subsection{Risk analysis}

\begin{lemma}[Stable ridge-intercept weights]\label{lem:pe-weights-stability}
Whenever $M_n\geq1$, with $C_w=1+d^2/\lambda$,
\begin{equation}\label{eq:stable-weight-bounds}
\sum_{e\in\mathcal E_n}w_e=1,
\qquad
\max_{e\in\mathcal E_n}|w_e|\leq\frac{C_w}{M_n},
\qquad
\sum_{e\in\mathcal E_n}|w_e|\leq C_w.
\end{equation}
\end{lemma}

\begin{proof}
The first identity follows from $\sum_e(s_e-\bar s)=0$. The pair screen implies $0\leq s_e\leq d$, so $0\leq\bar s\leq d$ and $|s_e-\bar s|\leq d$. Since $S_s+\lambda M_n\geq\lambda M_n$, \eqref{eq:pe-weights} gives
\begin{equation*}
|w_e|
\leq
\frac1{M_n}+\frac{d^2}{\lambda M_n}
=\frac{C_w}{M_n}.
\end{equation*}
Summing proves the final bound.
\end{proof}

\begin{lemma}[Conditional risk]\label{lem:pe-conditional risk-new}
There is $C<\infty$ such that, on every realized design with $M_n\geq1$,
\begin{equation}\label{eq:stable-conditional risk}
\left|
\E(\widetilde\sigma_{n,\mathrm{PE}}^2\mid X)-\sigma^2
\right|
\leq CH_n^{2\beta_b}h_n^2,
\qquad
\Var(\widetilde\sigma_{n,\mathrm{PE}}^2\mid X)
\leq\frac{C}{M_n}.
\end{equation}
\end{lemma}

\begin{proof}
For every retained pair, define the local unit contrast $u_{e,Q}=R_Qd_e/\sqrt{v_e}$, and let $u_e\in\R^n$ be its zero extension outside $Q$. Write $f=(f(X_1),\ldots,f(X_n))^\top$ and set
\begin{equation}\label{eq:pe-quadratic-form}
A_X=\sum_{e\in\mathcal E_n}w_eu_eu_e^\top.
\end{equation}
Then $\widetilde\sigma_{n,\mathrm{PE}}^2=Y^\top A_XY$. Since $\|u_e\|_2=1$, \Cref{lem:pe-weights-stability} gives
$\operatorname{tr}(A_X)=\sum_ew_e=1$. Let $\mu_e=u_e^\top f$. By \Cref{lem:pe-bias bound}, $|\mu_e|\leq CH_n^{\beta_b}h_n$. Therefore,
\begin{equation*}
\left|
\E(Y^\top A_XY\mid X)-\sigma^2
\right|
=
|f^\top A_Xf|
=
\left|\sum_ew_e\mu_e^2\right|
\leq CH_n^{2\beta_b}h_n^2,
\end{equation*}
where we used $\sum_e|w_e|\leq C_w$.

A retained cell contains at most $K=\binom{N_{\max}}2$ retained pairs. Let $A_{Q,X}$ be the block of $A_X$ corresponding to cell $Q$. Since $\|u_eu_e^\top\|_F=1$,
\begin{equation*}
\|A_{Q,X}\|_F
\leq
\sum_{e\in\mathcal E_n(Q)}|w_e|
\leq\frac{KC_w}{M_n}.
\end{equation*}
At most $M_n$ cell blocks contain a retained pair, and the blocks have disjoint supports. Hence
\begin{equation}\label{eq:stable-frobenius-bound}
\|A_X\|_F^2
=\sum_Q\|A_{Q,X}\|_F^2
\leq\frac{K^2C_w^2}{M_n}.
\end{equation}
Similarly,
$
A_{Q,X}f_Q
=
\sum_{e\in\mathcal E_n(Q)}w_e\mu_eu_{e,Q},
$
so
\begin{equation}\label{eq:stable-Af-bound}
\|A_Xf\|_2^2
\leq
\frac{CH_n^{2\beta_b}h_n^2}{M_n}
\leq\frac{C}{M_n}.
\end{equation}

Write $Y=f+\varepsilon$. Conditional independence, centering, the common conditional variance, and the fourth-moment bound imply
\begin{align*}
\Var(\varepsilon^\top A_X\varepsilon\mid X)
&=
\sum_i(A_X)_{ii}^2\Var(\varepsilon_i^2\mid X)
+4\sigma^4\sum_{i<j}(A_X)_{ij}^2
\leq C\|A_X\|_F^2.
\end{align*}
Also,
$\Var(f^\top A_X\varepsilon\mid X)=\sigma^2\|A_Xf\|_2^2$. Finally,
\begin{equation*}
Y^\top A_XY-\E(Y^\top A_XY\mid X)
=
2f^\top A_X\varepsilon
+
\{\varepsilon^\top A_X\varepsilon-\sigma^2\operatorname{tr}(A_X)\}.
\end{equation*}
Using $\Var(U+V\mid X)\leq2\Var(U\mid X)+2\Var(V\mid X)$ together with \eqref{eq:stable-frobenius-bound}--\eqref{eq:stable-Af-bound} proves the variance bound in \eqref{eq:stable-conditional risk}.
\end{proof}

\begin{proof}[Proof of \Cref{prop:stable-est-lr}]
Choose any fixed $N_{\max}\geq q+2$ and choose $\kappa$ and $v_0$ as in \eqref{eq:certificate-screen-constants-new}. Every cell in $\mathcal C_n(h_n)$ is then retained by the extension and contributes at least one eligible pair, so $M_n\geq|\mathcal C_n(h_n)|$.
Apply \Cref{lem:pe-availability} with $t_n=h_n$, and let $\mathcal H_n=\{M_n\geq cnh_n^d\}$. On $\mathcal H_n$, \Cref{lem:pe-conditional risk-new} gives
\begin{equation*}
\E\left[
(\widetilde\sigma_{n,\mathrm{PE}}^2-\sigma^2)^2
\mathrel{\Big|}X
\right]
\leq
\frac{C}{nh_n^d}
+CH_n^{4\beta_b}h_n^4.
\end{equation*}
On $\mathcal H_n^c\cap\{M_n\geq1\}$, \Cref{lem:pe-conditional risk-new} gives a uniform constant bound on the conditional mean squared error because $M_n\geq1$ and $H_n,h_n\leq1$. On $\{M_n=0\}$, the deterministic fallback also has bounded squared error. Integrating over the design and applying \Cref{lem:pe-availability} proves \eqref{eq:stable-generic-risk} for the unclipped estimator. Clipping cannot increase squared error because $\sigma^2\in[\underline v,\overline v]$, so the bound also holds for the clipped estimator.

Now take $h_n\asymp n^{-a_\star}$ with $a_\star=(d-4\beta_b)/\{d(d+4)\}$. Since $H_n\asymp n^{-1/d}$,
$
(nh_n^d)^{-1}
\asymp
n^{-4(\beta_b+1)/(d+4)},
H_n^{4\beta_b}h_n^4
\asymp
n^{-4(\beta_b+1)/(d+4)}.
$
Also $nh_n^d\asymp n^{4(\beta_b+1)/(d+4)}$, so the exponential term is negligible. This proves \eqref{eq:stable-optimal-rate}.
\end{proof}

\subsection{The close-pair estimator}

\label{cp}

This section 
studies the performance of the close pair construction.
Use the same bounded-aspect equal-volume partition, with
$J_n\asymp n/\lambda_0$ and $H_n=J_n^{-1/d}$.  In each cell, order the observations
by sample label and form the disjoint consecutive blocks
\begin{equation*}
(i_{Q,1},i_{Q,2}),
(i_{Q,3},i_{Q,4}),\ldots.
\end{equation*}
For a deterministic normalized bandwidth $\delta_n\in(0,1]$, retain a block
$e=(i,j)$ when $\|U_i-U_j\|_\infty\leq\delta_n$.  Let
$\mathcal A_n(\delta_n)$ be the retained collection and
$M_n(\delta_n)=|\mathcal A_n(\delta_n)|$.  When $M_n(\delta_n)>0$, set
\begin{equation}
\label{eq:direct-pair-estimator-definition}
\widetilde\sigma_{n,\mathrm{pair}}^2(\delta_n)
=
\frac{1}{M_n(\delta_n)}
\sum_{(i,j)\in\mathcal A_n(\delta_n)}
\frac{(Y_i-Y_j)^2}{2}.
\end{equation}
Use any deterministic fallback in $[\underline v,\overline v]$ when
$M_n(\delta_n)=0$, and clip to $[\underline v,\overline v]$ to obtain
$\widehat\sigma_{n,\mathrm{pair}}^2(\delta_n)$.

\begin{proposition}[Direct close-pair benchmark: sharp saturation rate]
\label{prop:raw-close-pair}
Assume $\beta_b>1$, $\beta_g>0$, $d>4$, and $\underline v<\overline v$. There are fixed constants
$\delta_0>0$ and $m_0<\infty$ such that, for all sufficiently large $n$ and
every deterministic $\delta_n\in(0,\delta_0]$ satisfying
$n\delta_n^d\geq m_0$,
\begin{equation}
\label{eq:direct-pair-generic-sharp-rate}
c\left\{
\frac{1}{n\delta_n^d}+H_n^4\delta_n^4
\right\}
\leq
\sup_{\theta\in\Theta(\beta_b,\beta_g,d)}
\E_\theta\left[\{\widehat\sigma_{n,\mathrm{pair}}^2(\delta_n)-\sigma^2\}^2
\right]
\leq
C\left\{
\frac{1}{n\delta_n^d}+H_n^4\delta_n^4
\right\}.
\end{equation}
Let
$\mathfrak B_n=\{\delta\in(0,\delta_0]:n\delta^d\geq m_0\}$.  Then
\begin{equation}
\label{eq:direct-pair-minimax-family}
\inf_{\delta\in\mathfrak B_n}
\sup_{\theta\in\Theta(\beta_b,\beta_g,d)}
\E_\theta
\left[
\{\widehat\sigma_{n,\mathrm{pair}}^2(\delta)-\sigma^2\}^2
\right]
\asymp n^{-8/(d+4)}.
\end{equation}
The minimizing order is
\begin{equation}
\label{eq:direct-pair-optimal-bandwidth}
\delta_n^{\mathrm{pair}}
\asymp
n^{-(d-4)/\{d(d+4)\}}.
\end{equation}
Thus the method saturates at smoothness one: its optimized rate does not improve
when $\beta_b>1$ increases.
\end{proposition}

The proof uses the following fixed-sample geometric fact.  Its directional
component is needed only for the lower bound.

\begin{lemma}[Count and directional spread of direct pairs]
\label{lem:direct-pair-geometry}
Under the uniform design, there are constants
$0<c_1<C_1<\infty$, $c_2>0$, and $\delta_0>0$ such that, whenever
$0<\delta\leq\delta_0$ and $\mu=n\delta^d\geq1$, an event $\mathcal E_n(\delta)$
with $\Pp\{\mathcal E_n(\delta)\}\geq c_2$ exists on which
\begin{equation}
\label{eq:direct-pair-count-spread-event}
c_1\mu\leq M_n(\delta)\leq C_1\mu
\end{equation}
and
\begin{equation}
\label{eq:direct-pair-directional-spread}
\frac{1}{M_n(\delta)}
\sum_{(i,j)\in\mathcal A_n(\delta)}
(X_{i1}-X_{j1})^2
\geq c_1H_n^2\delta^2.
\end{equation}
\end{lemma}

\begin{proof}
For $0<\delta\leq\delta_0$, choose a symmetric measurable set
$\mathcal V_\delta\subset[0,1]^d\times[0,1]^d$ with the following properties:
\begin{equation*}
|\mathcal V_\delta|\geq c\delta^d,
\qquad
\|u-v\|_\infty\leq\delta,
\qquad
|u_1-v_1|\geq\delta/2
\end{equation*}
for every $(u,v)\in\mathcal V_\delta$.  For example, one may take $u$ in a
fixed cube strictly inside $(0,1)^d$, take
$u_1-v_1\in[\delta/2,3\delta/4]$, take
$|u_j-v_j|\leq\delta/4$ for $j\geq2$, and then symmetrize.

Let $L_n(\delta)$ be the number of cells containing exactly two observations
whose normalized coordinates belong to $\mathcal V_\delta$.  Such a cell
contributes one retained consecutive pair.  For a fixed cell $Q$ under the
uniform design,
\begin{equation*}
\Pp(Q\text{ contributes to }L_n)
=
\binom n2J_n^{-2}(1-J_n^{-1})^{n-2}|\mathcal V_\delta|
\asymp\delta^d.
\end{equation*}
For two distinct cells $Q,Q'$, the joint contribution probability is at most a
constant multiple of $\delta^{2d}$, because it equals
$
\frac{n!}{2!2!(n-4)!}
J_n^{-4}(1-2J_n^{-1})^{n-4}|\mathcal V_\delta|^2.
$
Since $J_n\asymp n$, it follows that
\begin{equation*}
\E L_n(\delta)\geq c\mu,
\qquad
\E L_n(\delta)^2\leq C(\mu+\mu^2).
\end{equation*}
Paley--Zygmund therefore gives
$
\Pp\{L_n(\delta)\geq c\mu\}\geq c'>0.
$

Conditional on all cell counts, each candidate consecutive pair consists of two
independent uniform normalized points.  Its acceptance probability is at most
$C\delta^d$, and there are at most $n/2$ candidate pairs.  Hence
$\E M_n(\delta)\leq C\mu$.  Markov's inequality, with a sufficiently large
constant, shows that $M_n(\delta)\leq C_1\mu$ except on an event whose
probability is smaller than $c'/2$.  Intersecting these events gives an event of
probability at least $c_2>0$ on which
$c_1\mu\leq L_n(\delta)\leq M_n(\delta)\leq C_1\mu$.

Every pair counted by $L_n(\delta)$ satisfies
$|U_{i1}-U_{j1}|\geq\delta/2$.  Since the first side length of every cell is at
least a fixed multiple of $H_n$,
\begin{equation*}
\sum_{(i,j)\in\mathcal A_n(\delta)}
(X_{i1}-X_{j1})^2
\geq cL_n(\delta)H_n^2\delta^2.
\end{equation*}
Dividing by $M_n(\delta)\leq C_1\mu$ proves
\eqref{eq:direct-pair-directional-spread}.
\end{proof}

We also use a simple fact about clipping an average whose target lies strictly
inside the clipping interval.

\begin{lemma}[Clipped averages retain a variance-order risk lower bound]
\label{lem:clipped-average-variance-lower}
Let $W_1,W_2,\ldots$ be iid bounded random variables with
$\E W_1=v\in(\underline v,\overline v)$ and
$\Var(W_1)>0$.  Then, for all sufficiently large $m$,
\begin{equation}
\label{eq:clipped-average-variance-lower}
\E\left[
\left\{
\clip_{[\underline v,\overline v]}
\left(\frac1m\sum_{r=1}^mW_r\right)-v
\right\}^2
\right]
\geq \frac{c}{m}.
\end{equation}
\end{lemma}

\begin{proof}
Let $\Delta=\min\{v-\underline v,\overline v-v\}>0$ and
$\overline W_m=m^{-1}\sum_{r=1}^mW_r$.  On
$\{|\overline W_m-v|\leq\Delta\}$ clipping does nothing.  Since the $W_r$ are
bounded, Hoeffding's inequality gives
$\Pp(|\overline W_m-v|>\Delta)\leq2e^{-cm}$.  Therefore,
\begin{align*}
\E\left[
\{\clip_{[\underline v,\overline v]}(\overline W_m)-v\}^2
\right]
&\geq
\E\left[
(\overline W_m-v)^2
\ind\{|\overline W_m-v|\leq\Delta\}
\right]
\geq
\frac{\Var(W_1)}{m}-Ce^{-cm}
\geq \frac{c}{m}
\end{align*}
for all sufficiently large $m$.
\end{proof}

\begin{proof}[Proof of \Cref{prop:raw-close-pair}]
We first prove the upper bound.  A Poissonized cell contains exactly two points
whose normalized $\ell_\infty$ distance is at most $\delta_n$ with probability
at least $c\delta_n^d$, uniformly over the model class.  The restrictions to
different cells are independent.  The tilted de-Poissonization argument used in
the proof of \Cref{lem:pe-availability}, with a certificate containing two
points instead of $q+2$, therefore gives
\begin{equation}
\label{eq:direct-pair-availability}
\sup_{\theta\in\Theta(\beta_b,\beta_g,d)}
\Pp_\theta\{M_n(\delta_n)<cn\delta_n^d\}
\leq C\exp(-cn\delta_n^d).
\end{equation}

Conditionally on the covariates, for every retained pair,
\begin{equation*}
\E\left\{\frac{(Y_i-Y_j)^2}{2}\mathrel{\big|}X\right\}
=
\sigma^2+\frac{\{f(X_i)-f(X_j)\}^2}{2}.
\end{equation*}
Because $\beta_b>1$, $\|\nabla f\|_\infty\leq C$, and the physical distance of a
retained pair is at most $CH_n\delta_n$.  Hence the conditional bias of
\eqref{eq:direct-pair-estimator-definition} is at most
$CH_n^2\delta_n^2$.  The retained pairs are disjoint, so their statistics are
conditionally independent.  The uniform fourth-moment bound and boundedness of
$f$ imply a uniform bound on their conditional variances.  Therefore, on the
event $\{M_n(\delta_n)\geq cn\delta_n^d\}$,
\begin{equation*}
\E\left[
\{\widetilde\sigma_{n,\mathrm{pair}}^2(\delta_n)-\sigma^2\}^2
\mathrel{\big|}X
\right]
\leq
C\left\{
\frac{1}{n\delta_n^d}+H_n^4\delta_n^4
\right\}.
\end{equation*}
Clipping cannot increase squared error, and the clipped fallback has bounded
risk.  After increasing $m_0$ if necessary, the exponential term in
\eqref{eq:direct-pair-availability} is smaller than a constant multiple of
$(n\delta_n^d)^{-1}$.  This proves the upper bound in
\eqref{eq:direct-pair-generic-sharp-rate}.

We next prove the lower bound.  Work throughout under the uniform design and
write $\mu_n=n\delta_n^d$.  Fix
$v_\circ\in(\underline v,\overline v)$.  Since $M_4>\overline v^{\,2}$, choose
$A>0$ such that $\overline v<A^2<M_4/\overline v$.  Let the errors be iid,
independent of the design, with the bounded three-point law
\begin{equation*}
Q_{v_\circ}(\{-A\})=Q_{v_\circ}(\{A\})=\frac{v_\circ}{2A^2},
\qquad
Q_{v_\circ}(\{0\})=1-\frac{v_\circ}{A^2}.
\end{equation*}
This law is centered, has variance $v_\circ$ and fourth moment
$A^2v_\circ<A^2\overline v<M_4$, and thus satisfies the model's error conditions.

First take $f=0$.  Conditional on $X$ and on $M_n=m$, the retained-pair
statistics are iid copies of
$
W=\frac{(\varepsilon_1-\varepsilon_2)^2}{2}$.
They are bounded, satisfy $\E W=v_\circ$, and have strictly positive variance;
in fact,
$
\Var(W)=\frac12\{A^2v_\circ+v_\circ^2\}>0$.
On the event $\mathcal E_n(\delta_n)$ from
\Cref{lem:direct-pair-geometry}, one has
$c\mu_n\leq M_n\leq C\mu_n$.  After increasing $m_0$ if necessary,
\Cref{lem:clipped-average-variance-lower} gives
\begin{equation}
\label{eq:direct-pair-variance-lower}
\sup_{\theta\in\Theta(\beta_b,\beta_g,d)}
\E_\theta
\left[
\{\widehat\sigma_{n,\mathrm{pair}}^2(\delta_n)-\sigma^2\}^2
\right]
\geq \frac{c}{\mu_n}.
\end{equation}

To obtain the geometric bias term, take $f(x)=a_0x_1$ on $[0,1]^d$. To satisfy the extension condition even when $U$ is unbounded, choose $\chi\in C_c^\infty(U)$ equal to one on a neighborhood of $[0,1]^d$ and use the extension $x\mapsto a_0x_1\chi(x)$. Its $C^{\beta_b}(U)$ norm is finite and proportional to $a_0$, so a sufficiently small fixed $a_0>0$ places this extension in $\cH^{\beta_b}(L;U)$.  On $\mathcal E_n(\delta_n)$, the conditional bias
$b_X$ of the unclipped estimator satisfies, by
\eqref{eq:direct-pair-directional-spread},
\begin{equation}
\label{eq:direct-pair-affine-bias-bounds}
cH_n^2\delta_n^2
\leq
b_X
=
\frac{a_0^2}{2M_n}
\sum_{(i,j)\in\mathcal A_n(\delta_n)}
(X_{i1}-X_{j1})^2
\leq
CH_n^2\delta_n^2.
\end{equation}
The upper bound uses the acceptance condition and the bounded aspect ratio.
Under the bounded error law, every retained-pair statistic is uniformly
bounded.  Conditional independence of the disjoint pairs therefore gives
\begin{equation}
\label{eq:direct-pair-affine-variance-upper}
\Var\{\widetilde\sigma_{n,\mathrm{pair}}^2(\delta_n)\mid X\}
\leq \frac{C}{M_n}
\leq \frac{C}{\mu_n}
\end{equation}
on $\mathcal E_n(\delta_n)$.

Let $r_{b,n}=H_n^4\delta_n^4$.  If
$\mu_n^{-1}\geq c_\star r_{b,n}$ for a sufficiently small fixed
$c_\star>0$, then \eqref{eq:direct-pair-variance-lower} is already at least a
constant multiple of
$\max\{\mu_n^{-1},r_{b,n}\}$.  Suppose instead that
$\mu_n^{-1}<c_\star r_{b,n}$.  Choose $c_\star$ so small that
\eqref{eq:direct-pair-affine-variance-upper} and Chebyshev's inequality imply,
conditionally on every design in $\mathcal E_n(\delta_n)$,
\begin{equation*}
\Pp\left(
\left|
\widetilde\sigma_{n,\mathrm{pair}}^2(\delta_n)
-\E\{\widetilde\sigma_{n,\mathrm{pair}}^2(\delta_n)\mid X\}
\right|
\leq \frac{c}{2}H_n^2\delta_n^2
\mathrel{\Big|}X
\right)
\geq\frac12,
\end{equation*}
where $c$ is the lower-bias constant in
\eqref{eq:direct-pair-affine-bias-bounds}.  Since
$H_n^2\delta_n^2\to0$ and $v_\circ$ lies strictly inside the clipping interval,
for all sufficiently large $n$ the estimator is not clipped on this event and
its distance from $v_\circ$ is at least
$cH_n^2\delta_n^2/2$.  Thus
\begin{equation*}
\E\left[
\{\widehat\sigma_{n,\mathrm{pair}}^2(\delta_n)-v_\circ\}^2
\mathrel{\Big|}X
\right]
\geq cH_n^4\delta_n^4
\end{equation*}
on $\mathcal E_n(\delta_n)$.  Since this design event has probability bounded
away from zero, the same lower bound holds unconditionally.  Combining the two
cases with \eqref{eq:direct-pair-variance-lower} gives
\begin{equation*}
\sup_{\theta\in\Theta(\beta_b,\beta_g,d)}
\E_\theta
\left[
\{\widehat\sigma_{n,\mathrm{pair}}^2(\delta_n)-\sigma^2\}^2
\right]
\geq
c\max\left\{
\frac{1}{n\delta_n^d},H_n^4\delta_n^4
\right\},
\end{equation*}
which is equivalent, up to a factor of two, to the lower bound in
\eqref{eq:direct-pair-generic-sharp-rate}.

Finally, because $H_n\asymp n^{-1/d}$, balancing the two terms gives
$\frac{1}{n\delta_n^d}
\asymp
H_n^4\delta_n^4
$ if and only if
$\delta_n\asymp n^{-(d-4)/\{d(d+4)\}}$.  At this choice, both terms are of
order $n^{-8/(d+4)}$, proving \eqref{eq:direct-pair-minimax-family} and
\eqref{eq:direct-pair-optimal-bandwidth}.
\end{proof}

\end{document}